\documentclass[11pt,a4paper,reqno]{amsart}

\usepackage[a4paper,left=1in,right=1in,top=1.4in,bottom=1in]{geometry}
\usepackage{amsmath,amssymb,amsfonts,mathtools}
\usepackage{microtype}
\usepackage{hyperref}
\usepackage{enumitem}
\usepackage{upgreek}
\usepackage{bm}
\hypersetup{
  colorlinks=true,
  linkcolor=blue,
  citecolor=blue,
  urlcolor=blue
}

\newtheorem{theorem}{Theorem}[section]
\newtheorem{corollary}[theorem]{Corollary}
\newtheorem{lemma}[theorem]{Lemma}
\newtheorem{proposition}[theorem]{Proposition}
\theoremstyle{definition}

\theoremstyle{remark}
\newtheorem{remark}[theorem]{Remark}

\numberwithin{equation}{section}

\newcommand{\C}{\mathbb C}
\newcommand{\R}{\mathbb R}
\newcommand{\F}{\mathcal F}
\newcommand{\dd}{\,d}
\newcommand{\e}{\mathrm e}
\newcommand{\Real}{\operatorname{Re}}
\newcommand{\ip}[2]{\left\langle #1,#2\right\rangle}

\allowdisplaybreaks[4] 

\title[Gaussian Integral Operator Between Fock Spaces]{Boundedness and Compactness  of   Gaussian Integral Operators between Fock Spaces}

\author{Yuanqi Sang}
\address{School of Mathematics, Southwestern University of Finance and Economics,
Chengdu, 611130, China}
\email{sangyq@swufe.edu.cn}

\author{Liankuo Zhao}
\address{School of Mathematical Sciences, Shanxi Normal University,
  Taiyuan, 030031, China}
\email{lkzhao@sxnu.edu.cn}

\date{}
\makeatletter
\@ifundefined{subjclassname@2020}{%
  \@namedef{subjclassname@2020}{\textup{2020} Mathematics Subject Classification}%
}{}
\makeatother

\begin{document}
\begin{abstract}
We study Gaussian integral operators
\[(Tf)(z)=\frac1\pi\int_{\mathbb C} e^{\alpha z^2+\beta\overline w^{\,2}+\gamma z\overline w+\eta z+\delta\overline w} f(w)e^{-|w|^2}\,dA(w)\]
acting between Fock spaces $\mathcal F^p$ and $\mathcal F^q$, where $0<p,q\leq\infty$. We give complete characterizations of their boundedness and compactness in terms of explicit conditions on the parameters of the Gaussian kernel.
The proofs are based on a unified approach combining reproducing-kernel estimates, atomic decompositions, and sampling sequences, together with a reduction to diagonal operators on sequence spaces.
We also study Schatten-class membership on $\mathcal F^2$ and compute the Hilbert--Schmidt norm explicitly.
\end{abstract}

\maketitle

\section{Introduction}

Let \(\C\) be the complex plane and \(d A\) be the Lebesgue area measure on \(\C\). For \(0<p< \infty\), the Fock space \(\F^p\) consists of all entire functions \(f\) with (quasi-)norm
\[
\|f\|_p=\left(\frac{p}{2\pi}\int_{\C}
  |f(z)\e^{-|z|^2/2}|^p d A(z)\right)^{\frac{1}{p}}.
\]
For \(p=\infty\), the Fock space \(\F^\infty\) consists of all entire functions \(f\) with norm
\[
\|f\|_\infty= \sup_{z\in\C}|f(z)\e^{-|z|^2/2}|.
\]
The Fock spaces arise from Bargmann's realization of quantum mechanical
state spaces by entire functions \cite{Bargmann}. They provide a setting in
which reproducing kernels, Gaussian integration, and concrete operator
models can be studied together. We refer to \cite{Zhu2012} for their function
theory and operator theory.

In \cite{DZ1, DZ2}, Dong and Zhu  introduced and studied a two-parameter family of integral
operators on the Fock space \(\F^2\).
Motivated by their work, in this paper, we consider  a class of integral operators \(T\) with more parameters,
\[
(Tf)(z)=\frac1\pi\int_{\C}
  \exp\bigl(\alpha z^2+\beta\bar w^2+\gamma z\bar w+\eta z+\delta\bar w\bigr)
  f(w)\e^{-|w|^2} d A(w),\qquad z\in \C,
\]
where
\(\alpha,\beta,\gamma,\eta,\delta\in\C.\)
The integral operator studied in \cite{DZ1} corresponds, up to a nonzero scalar factor, to \(T\) with  the parameters
\(
 \alpha=t/(2s), \beta=-t/(2s),
 \gamma=1/s, \eta=\delta=0,
\) where \(s\ne 0\).
On \(\F^2\), they characterize boundedness and compactness.
Their sequel \cite{DZ2} shows that boundedness on \(\F^p\) (\(0<p\le\infty\)) is independence of the
exponent, obtains singular values,
and proves membership in every Schatten class in the compact case. We will study,  in a unified way,  the boundedness and compactness of \(T\) from \(\F^p\) to \(\F^q\)
for all \(0<p,q\le\infty\).

The study of integral operators \(T\) can be traced back to Fourier integrals \cite{Neretin}. Such operators are called Gaussian integral operators and belong to a broader analytic tradition.
Lieb \cite{Lieb} studies sharp bounds and Gaussian extremizers for integral
operators with Gaussian kernels on real \(L^p\) spaces, notably for
\(1<p\le q<\infty\). The phase-space setting provides a further motivation.
Quadratic and linear
terms are natural in the study of oscillator evolution and phase-space
displacements, and Gaussian kernels also occur in the oscillator
semigroup \cite{Howe}. Under the Bargmann transform, canonical Gaussian kernels represent linear
canonical transformations; see \cite{DZ1,Folland}.
There is also a direct connection with weighted composition operators on Fock spaces. When
\(\beta=0\), the reproducing formula gives
\[
 (Tf)(z)=\exp(\alpha z^2+\eta z)\,f(\gamma z+\delta), \qquad z\in\C.
\]
This is a very important class of weighted composition operators. Indeed,  unitary, normal,  and complex symmetric weighted composition operators on \(\F^2\) all have this form \cite{HK, SZ, Zhao2014, Zhao2017}.

For the convenience of presenting our main results, we introduce some notation.

Let \(\mathbb R\) be the set of real numbers. For \(z\in\C\), \(\Real z\) and \(\operatorname{Im} z\) are the real part and imaginary part of \(z\) respectively.
For  \(|\alpha|<1/2\), \(|\beta|<1/2\), denote
\begin{align}\label{eq:NalphaMalpha}
N_\alpha=\bar\beta+\frac{\alpha\bar\gamma^2}{1-4|\alpha|^2},\quad
  M_\alpha=\bar\delta+\frac{\bar\gamma\eta+2\alpha\bar\gamma\bar\eta}{1-4|\alpha|^2}
\end{align}
and
\begin{eqnarray}\label{eq:Thetaalpha}
\begin{array}{c}
\Theta_\alpha(w)=\frac{1}{2}\left(\frac{|\gamma|^2}{1-4|\alpha|^2}-1
  \right)|w|^2+\Real\left(N_\alpha w^2+M_\alpha w\right),\\[1em]
\mathcal R(w)=\e^{\Theta_\alpha(w)},
\end{array}
\qquad w\in\C.
\end{eqnarray}
Similarly, denote
\begin{align}\label{eq:NbetaMbeta}
  N_\beta=\alpha+\frac{\bar\beta\gamma^2}{1-4|\beta|^2},\qquad
  M_\beta=\eta+\frac{\gamma\bar\delta+2\bar\beta\gamma\delta}{1-4|\beta|^2}
\end{align}
and
\begin{eqnarray}\label{eq:Thetabeta}
\begin{array}{c}
  \Theta_\beta(z)=\frac{1}{2}\left(\frac{|\gamma|^2}{1-4|\beta|^2}-1\right)|z|^2
  +\Real\left(N_\beta z^2+M_\beta z\right),\\[1em]
\mathcal C(z)=\e^{\Theta_\beta(z)},
\end{array}
\qquad z\in\C.
\end{eqnarray}

The main results read
as follows.

\begin{theorem}\label{thm:fp-fq}
Assume \(\gamma\neq0\). Let \(0<p\le q\le\infty\). Then the following are
equivalent:
\begin{itemize}
\item[\textup{(i)}] \(T:\F^p\to\F^q\) is bounded;
\item[\textup{(ii)}]  \(|\alpha|<1/2,\quad  |\beta|<1/2\) and \(\sup_{w\in\C}\mathcal R(w)<\infty\);
\item[\textup{(iii)}] \(|\alpha|<1/2,\quad  |\beta|<1/2\) and \(\sup_{z\in\C}\mathcal C(z)<\infty\).
\end{itemize}
\end{theorem}

\begin{theorem}\label{thm:fp-fq compact}
Assume \(\gamma\neq0\). Let \(0<p\le q\le\infty\). Then the following are
equivalent:
\begin{enumerate}
\item[\textup{(i)}] \(T:\F^p\to\F^q\) is compact;
\item[\textup{(ii)}] \(|\alpha|<1/2,\quad  |\beta|<1/2\) and \(\mathcal R(w)\to 0\) as \(|w|\to\infty\);
\item[\textup{(iii)}]  \(|\alpha|<1/2,\quad  |\beta|<1/2\) and \(\mathcal C(z)\to 0\) as \(|z|\to\infty\).
\end{enumerate}
Equivalently, either of the following strict inequalities hold:
\begin{equation}\label{eq:strict-inequality}
\begin{split}
|\alpha|<1/2,\quad  |\beta|<1/2,\quad
 \frac{|\gamma|^2}{1-4|\alpha|^2}
 +2|N_\alpha|<1,\\
|\alpha|<1/2,\quad  |\beta|<1/2,\quad
 \frac{|\gamma|^2}{1-4|\beta|^2}
 +2|N_\beta|<1.
\end{split}
\end{equation}

\end{theorem}

\begin{theorem}\label{thm:fp-fq-compact-q-less-p}
Assume \(\gamma\neq0\). Let \(0<q<p\le\infty\). Then the following are
equivalent:
\begin{enumerate}
\item[\textup{(i)}]  \(T:\F^p\to\F^q\) is compact;
\item[\textup{(ii)}]  \(T:\F^p\to\F^q\) is bounded;
\item[\textup{(iii)}]  \(|\alpha|<1/2,\quad  |\beta|<1/2\) and \(\mathcal R(w)\to 0\) as \(|w|\to\infty\);
\item[\textup{(iv)}]  \(|\alpha|<1/2,\quad  |\beta|<1/2\) and \(\mathcal C(z)\to 0\) as \(|z|\to\infty\).
\end{enumerate}
Equivalently, the strict inequalities \eqref{eq:strict-inequality} displayed in
Theorem~\ref{thm:fp-fq compact} hold.
\end{theorem}

Define
\[
  P(z)=\exp\bigl(\alpha z^2+\eta z\bigr),\ z\in\C; \qquad
Q(w)=\exp\bigl(\bar\beta w^2+\bar\delta w\bigr),\ w\in\C.
\]

\begin{theorem}\label{thm:gamma=0}
Assume \(\gamma=0\). Let \(0<p, q\le\infty\). Then the following are
equivalent:
\begin{enumerate}
\item[\textup{(i)}]  \(T:\F^p\to\F^q\) is compact;
\item[\textup{(ii)}]  \(T:\F^p\to\F^q\) is bounded;
\item[\textup{(iii)}] \(P\in \F^q\), with \(|\beta|<1/2\) for \(1<p\le \infty\) and \(Q\in \F^\infty\) for \(0<p\le 1\).
\end{enumerate}
\end{theorem}

The paper is organized as follows. Section~2 collects the Fock space
tools and Gaussian estimates. Section~3 establishes the matrix reduction
and the diagonal multiplier criterion. Section~4 proves
Theorems~\ref{thm:fp-fq}--\ref{thm:gamma=0}. Section~5 treats Schatten
class membership and the Hilbert--Schmidt norm on \(\F^2\).
Throughout this paper, compactness means that bounded sets
have relatively compact images in the topology induced by the norm or
quasi-norm.

\section{Preliminaries}

\subsection{Reproducing kernels, duality}

The reproducing kernel and normalized reproducing kernel in \(\F^2\) are
\[K_w(z)=\e^{ z\overline w},\qquad k_w(z)=\e^{ z\overline w- |w|^2/2}.\]
Moreover, we have
\[\|k_w\|_p=1\] for  \(0<p\le \infty\) \cite[Lemma 2.33, p.64]{Zhu2012}.


\begin{lemma}\cite[Corollary 2.22, p.52; Corollary 2.8, p.38]{Zhu2012}\label{lem:reproducing}
For \(0< p\le \infty\), \(f\in\F^p\),
\[f(z)=\frac{1}{\pi}\int_{\C}f(w)\overline{K_z(w)}\e^{-|w|^2}dA(w),\qquad z\in \C.\]
\[|f(z)|\le \|f\|_p\e^{|z|^2/2},\qquad z\in \C.\]
\end{lemma}

\begin{lemma}\label{lem:inclusion}
If \(0<p_1<p_2\le\infty\), then \(\F^{p_1}\subset\F^{p_2}\), and the inclusion is continuous.
\end{lemma}
\begin{proof}
For \(p_2<\infty\), this is precisely \cite[Theorem 2.10, p.39]{Zhu2012}. For \(p_2=\infty\), it follows from Lemma~\ref{lem:reproducing}.
\end{proof}

\begin{lemma}\cite[Corollary 2.25, p.57]{Zhu2012}\label{lem:duality}
For \( 0 < p < \infty \), the dual space of \( \F^p \) can be identified with \( \F^{p'} \) under the integral pairing
\[
\langle f, g \rangle = \frac{1}{\pi} \int_{\mathbb{C}} f(z) \overline{g(z)} \e^{- |z|^2}dA(z),\qquad f\in\F^p, \quad g\in\F^{p'},
\]
where \(p'=\infty\) if \(0<p\le1\), \(p'=p/(p-1)\) if \(1<p<\infty\).
\end{lemma}

\subsection{Lattices and atomic decomposition}

Let \(\mathbb{Z}\) denote the set of integers. For \(r>0\), let
\[r\mathbb{Z}^2=\{rm+irn:\ m, n\in\mathbb{Z}\}\]
be the lattice generated by \(0, r\) and \(ir\). Then for any \(\xi\in\C\), there exists a \(\lambda_\xi\in r\mathbb{Z}^2\) such that
\begin{align}\label{eq:nearest-lattice-point}
|\xi-\lambda_\xi|=\min_{\lambda\in r\mathbb{Z}^2}|\xi-\lambda|\le \frac{r}{\sqrt{2}}.
\end{align}

\begin{lemma}\label{lem: lattices}
Let \(\Lambda=r\mathbb{Z}^2\). For any positive number \(\sigma\), there exist positive constants \(c_{r, \sigma}\), \(C_{r, \sigma}\) such that
\[
c_{r, \sigma}\le \sum_{\lambda\in\Lambda}e^{-\sigma|\lambda-\xi|^2}\le C_{r, \sigma}, \qquad \xi\in \C.
\]
\end{lemma}
\begin{proof}
The upper bound follows from \cite[Lemma 1.12, p.9]{Zhu2012}, and the lower
bound follows by retaining a nearest lattice point. That is, for any \(\xi\in\C\), choose the nearest lattice point \(\lambda_\xi\) to \(\xi\). Then by \eqref{eq:nearest-lattice-point},
\[
\sum_{\lambda\in\Lambda}e^{-\sigma|\lambda-\xi|^2}\geq e^{-\sigma|\lambda_\xi-\xi|^2 }\ge e^{-\sigma r^2/\sqrt{2}}.\qedhere
\]
\end{proof}


The following result follows from Lemma 1.14 and its proof in \cite[p.10]{Zhu2012}.

\begin{lemma}\label{lem: lattices decomposition}
Let \(\Lambda=r\mathbb{Z}^2\). For any positive number \(K\), let \(k=\lfloor \frac{K}{r}\rfloor+1\) and \(N=(k+1)^2\). Then
\begin{enumerate}
\item[{\rm{(i)}}]  \(\Lambda\) can be decomposed into the disjoint union of \(N\) sublattices,
\[\Lambda=\Lambda_1\cup\cdots\cup\Lambda_N,\]
and the distance between any two points in each of the sublattices is at least \(K\).
\item[{\rm{(ii)}}]  For any \(0<s<\infty\), there exists a positive constant \(C_{K}\) such that
\[\sup_{\lambda\in\Lambda_j}\sum_{\lambda'\in\Lambda_j, \lambda'\neq \lambda}\e^{-s|\lambda-\lambda'|^2}\le C_{K}, \qquad 1\le j\le N, \]
and \(\lim_{K\to\infty}C_{K}=0.\)
\end{enumerate}

\end{lemma}

The following result provides an atomic decomposition of \(\F^p\)   and  shows that, for some \(r>0\), \(r\mathbb Z^2\) is a sampling sequence for \(\F^p\).

\begin{lemma}\label{lem:dense} Let \( 0 < p \le \infty \). There exists a positive constant \( r_0 \) such that for any \( 0 < r < r_0 \) and \(\Lambda=r\mathbb{Z}^2\), there  exist  positive constants \(c_{r, p}, C_{r, p} \) (independent of \( f \)) such that the two statements hold.
\begin{enumerate}
\item[{\rm{(i)}}] \cite[Theorem 2.34, p.65]{Zhu2012} The space \( \F^p \) consists exactly of the following functions:
\begin{equation}\label{eq:dense}
f(z) = \sum_{\lambda \in \Lambda} \zeta_\lambda k_\lambda(z),
\end{equation}
where  \( \{\zeta_\lambda : \lambda \in \Lambda\} \in \ell^p(\Lambda) \). Moreover,
\[c_{r, p} \|f\|_{p} \leq \inf \| \{\zeta_\lambda\} \|_{\ell^p(\Lambda)} \leq C_{r, p} \|f\|_{p}\]
for all \( f \in \F^p \), where the infimum is taken over all sequences \( \{\zeta_\lambda: \lambda \in \Lambda\} \) that give rise to the decomposition (not unique) in \eqref{eq:dense}.
\item[{\rm{(ii)}}]  \cite[Corollary 4.37, p.175]{Zhu2012} For any \(f\in \F^p\),
\[
c_{r, p}\|\{f(\lambda)e^{-|\lambda|^2/2}\}_{\lambda\in\Lambda}\|_{\ell^p(\Lambda)}\le \|f\|_p\le C_{r, p}\|\{f(\lambda)e^{-|\lambda|^2/2}\}_{\lambda\in\Lambda}\|_{\ell^p(\Lambda)}.
\]
\end{enumerate}
\end{lemma}

\begin{remark}\label{rem:SpRp}
Let \(0<p\le\infty\).
Based on Lemma \ref{lem:dense}, for some \(r>0\) and \(\Lambda=r\mathbb{Z}^2\), we define the  operator \(S_p\) from \(\ell^p(\Lambda)\) to \(\F^p\) as
\[S_p\zeta=\sum_{\lambda\in\Lambda}\zeta_\lambda k_\lambda,\qquad \zeta=\{\zeta_\lambda\}_{\lambda\in \Lambda}\in \ell^p(\Lambda).\]
Then \(S_p\) is bounded and for any \(f\in\F^p\), there exists \(\zeta\in \ell^p(\Lambda)\) such that \(f=S_p\zeta\) and
\[\|\zeta\|_{\ell^p(\Lambda)}\le C_{r, p}\|f\|_p.\]
We also define the operator \(R_p\) from \(\F^p\) to \(\ell^p(\Lambda)\) as
\[R_p f=\{f(\lambda)e^{-|\lambda|^2/2}\}_{\lambda\in\Lambda}.\]
Then \(R_p\) is both bounded and bounded below.
\end{remark}

\subsection{Related integrals and functions}

The following integral formulas are elementary, see \cite[Lemma 2 (2)]{DZ1}.

\begin{lemma}\label{lem:Gaussian integral}
If \(|b|<1\), then
\[\frac{1}{\pi}\int_{\C}e^{\overline{b}\overline{z}^{2}
+cz+\bar{d}\overline{z}-|z|^{2}}d A(z)=
\e^{\bar{b} c^{2}+\bar{d}c}.\]
If
\(|b|<1/2\), then
\[\left\|\e^{bz^2+cz}\right\|_p^p=\frac1{\sqrt{1-4|b|^2}}\exp\left(\frac{p}{2}
   \frac{|c|^2+2\operatorname{Re}(\overline b c^2)}{1-4|b|^2}\right),\qquad 0<p<\infty.\]
\end{lemma}

\begin{lemma}\label{lem:quadratic-integrability}
Let
\[F(z)=\exp\left(\frac a2|z|^2+\Real(bz^2+cz)\right),
 \qquad a\in\mathbb R,\quad b,c\in\C,\]
and let \(\Lambda\) be any lattice in \(\C\). For one, and hence for
every, \(0<p<\infty\), the following are equivalent:
\begin{enumerate}
\item[{\rm{(i)}}] \(F\in L^p(\C,d A)\);
\item[{\rm{(ii)}}] \(\{F(\lambda)\}_{\lambda\in\Lambda}\in\ell^p(\Lambda)\);
\item[{\rm{(iii)}}] \(F(z)\to0\) as \(|z|\to\infty\);
\item[{\rm{(iv)}}] \(F(\lambda)\to0\) as \(|\lambda|\to\infty\), \(\lambda\in\Lambda\);
\item[{\rm{(v)}}] \(a/2+|b|<0\).
\end{enumerate}
\end{lemma}

\begin{proof}
Let \(b=|b|e^{-i2\theta}\), \(z=x+iy\). Then
\[F(e^{i\theta}z)=\exp\left((\frac{a}{2}+|b|)x^2+(\frac{a}{2}-|b|)y^2+c_1x-c_2y\right),\]
where \(c_1=\Real(ce^{i\theta})\), \(c_2=\operatorname{Im}(ce^{i\theta}).\) Hence
\begin{align*}
&\ \int_{\C}|F(z)|^pd A(z)=\int_{\C}|F(e^{i\theta}z)|^pd A(z)\\
=&\ \int_{\R^2}\exp\left(p(\frac{a}{2}+|b|)x^2+p(\frac{a}{2}-|b|)y^2+pc_1x-pc_2y\right)d xd y\\
=&\ \int_{\R}\exp\left(p(\frac{a}{2}+|b|)x^2+pc_1x\right)d x\int_{\R}\exp\left(p(\frac{a}{2}-|b|)y^2-pc_2y\right)d y.
\end{align*}
Therefore, \(\int_{\C}|F(z)|^pd A(z)<\infty\) if and only if
\[\int_{\R}\exp\left(p(\frac{a}{2}+|b|)x^2+pc_1x\right)d x<\infty,\qquad \int_{\R}\exp\left(p(\frac{a}{2}-|b|)y^2-pc_2y\right)d y<\infty.\]
Obviously, with the fact \(a/2-|b|\le a/2+|b|\), the formulas hold above if and only if
\[\frac{a}{2}+|b|<0.\]
This gives the equivalence of the statements (i) and (v).

If \(a/2+|b|<0\), then
\[\lim_{|z|\to\infty}|F(z)|\le\lim_{|z|\to\infty}\exp\left((\frac{a}{2}+|b|)|z|^2+|cz|\right)=0,\]
Hence \(F(z)\to 0\) as \(|z|\to\infty\).
Conversely, if  \(F(z)\to 0\) as \(|z|\to\infty\), then
\[\lim_{|x|\to\infty}\exp\left((\frac{a}{2}+|b|)x^2+c_1x\right)=\lim_{|x|\to\infty, y=0}|F(e^{i\theta}z)|=\lim_{|z|\to\infty}|F(z)|=0.\]
It follows that \(a/2+|b|<0.\) The equivalence of the statements (iii) and (v) follows.

Next, we give the equivalence of (ii), (iv) and (v).

If \(a/2+|b|<0\), then
\begin{align*}
|F(\lambda)|\le&\exp\left((\frac{a}{2}+|b|)|\lambda|^2+|c\lambda|\right)\\
\le& \exp\Bigl((\frac{a}{2}+|b|)|\lambda|^2+\frac12\bigl(\bigl|\frac{a}{2}+|b|\bigr||\lambda|^2
+\frac{|c|^2}{\bigl|a/2+|b|\bigr|}\bigr)\Bigr)\\
= &\exp\Bigl(\frac12(\frac{a}{2}+|b|)|\lambda|^2
+\frac12\frac{|c|^2}{\bigl|a/2+|b|\bigr|}\bigr)\Bigr).
\end{align*}
It follows from Lemma \ref{lem: lattices} that \(\sum_{\lambda\in\Lambda}|F(\lambda)|^p<\infty.\) Hence (v) implies (ii). Obviously, (ii) implies (iv).

Assume that (iv) holds.
Let \(\Omega=e^{-i\theta}\Lambda\), then  \(\Omega\) is also a lattice. We have
\[|F(\lambda)|=|F(e^{i\theta}\omega)|=\exp\left((\frac{a}{2}+|b|)x^2+c_1x+(\frac{a}{2}-|b|)y^2-c_2y\right), \quad \omega=x+iy\in\Omega.\]
We use the following fact about lattices over \(\C\). There exist \(L_\Omega>0\) and sequences \(\{\omega_{n}\}\) such that (1)
\(\{\omega_{n}=x_{n}+iy_{n}\in\Omega:\ \lim_{n\to\infty}x_{n}=\infty,\ |y_n|\le L_\Omega\}\);
(2) \(\{\omega_{n}=x_{n}+iy_{n}\in\Omega:\ \lim_{n\to\infty}x_{n}=-\infty,\ |y_n|\le L_\Omega\}\).
If \(a+2|b|>0\), taking \(\{\omega_n\}\in \Omega\) as (1), then
\[|F(\lambda)|\ge \inf_{|y|\le L_\Omega}\exp\left((\frac{a}{2}-|b|)y^2-c_2y\right)\exp\left((\frac{a}{2}+|b|)x_n^2+c_1x_n\right)\to \infty,\quad n\to\infty.\]
If \(a+2|b|=0\), \(c_1\neq 0\),  taking \(\{\omega_n\}\in \Omega\) as (1) or (2), depending on \(c_1>0\) or \(c_1<0\), then
\[|F(\lambda)|\ge \inf_{|y|\le L_\Omega}\exp\left((\frac{a}{2}-|b|)y^2-c_2y\right)\exp\left(c_1x_n\right)\to \infty,\quad n\to\infty.\]
If \(a+2|b|=0\), \(c_1= 0\),  taking \(\{\omega_n\}\in \Omega\) as (1) or (2), then
\[|F(\lambda)|\ge \inf_{|y|\le L_\Omega}\exp\left((\frac{a}{2}-|b|)y^2-c_2y\right)>0.\]
Therefore,
(iv) implies that (v).
\end{proof}

The following corollaries are well-known. For completeness, we outline the proofs.

\begin{corollary}\label{cor:nonzero in fp}
Let \(f(z)=\e^{bz^2+cz}\) with \(b, c\in \C\).
\begin{enumerate}
\item[{{\rm (i)}}] \(f\in \F^p\) for \(0<p< \infty\) if and only if \(|b|<\frac{1}{2}\).
\item[{{\rm (ii)}}] If \(|b|<1/2\), then \(f\in\F^\infty\); if \(f\in\F^\infty\), then \(|b|\le \frac{1}{2}.\)
\end{enumerate}
\end{corollary}
\begin{proof}
For \(0<p<\infty\), \(f\in \F^p\)  if and only if
\[
F(z)=\exp\bigl(-\frac{1}{2}|z|^2+\Real(bz^2+cz)\bigr)\in L^p(\C, dA).\] Hence assertion (i) follows from Lemma \ref{lem:quadratic-integrability}.

Assertion (ii) follows from a direct computation by the definition of \(\F^\infty\). Also, the first part in (ii) follows from (i) and Lemma \ref{lem:inclusion}.
\end{proof}

\begin{corollary}\label{cor:functional}
Let \(h(z)=\e^{bz^2+cz}\) with \(|b|<1/2\). Then \(h\) induces a continuous functional on \(\F^p\) for any \(0<p\le \infty\) under the integral pairing
\[
\langle f, h \rangle = \frac{1}{\pi} \int_{\mathbb{C}} f(z) \overline{h(z)} \e^{- |z|^2}dA(z),\qquad f\in\F^p.
\]
\end{corollary}
\begin{proof}
Since  \(|b|<1/2\), by Corollary \ref{cor:nonzero in fp}, \(h\in\F^{t}\) for any \(0<t\le \infty\).

When \(0<p<\infty\), the assertion  follows from Lemma \ref{lem:duality}.
When \(p=\infty\), for any \(f\in\F^\infty\),
\(
|\langle f, h \rangle |\leq 2\|f\|_\infty\|h\|_1.
\)
Hence the conclusion holds for \(p=\infty\).
\end{proof}

\subsection{Action on reproducing kernels}\label{subsec:action}

In this subsection, we assume throughout that
\(|\alpha|<\frac{1}{2}, |\beta|<\frac{1}{2}\).

Denote
\[H(z, w)=\exp\bigl(\alpha z^2+\beta\bar w^2+\gamma z\bar w+\eta z+\delta\bar w\bigr),\qquad z , w\in \C.\]
Then
\[
Tf(z)=\frac1\pi\int_{\C}H(z, w)f(w)\e^{-|w|^2}d A(w),\qquad z\in \C.
 \]
 Define
 \[
 T^\#f(w)=\frac1\pi\int_{\C}\overline{H(z, w)}f(z)\e^{-|z|^2}d A(z),\qquad w\in \C.
 \]
 Then we see that the parameters (\(\bar{\beta}, \bar{\alpha}, \bar{\gamma}, \bar{\delta}, \bar{eta}\)) in the definition of \(T^\#\) correspond to the parameters (\(\alpha, \beta, \gamma, \eta, \delta\)) in the definition of \(T\). In what follows, we shall always use this correspondence to obtain the corresponding results for \(T^\#\) from those for \(T\).

Since \(|\beta|<1/2\), for any \(z\in\C\), by Corollary \ref{cor:functional},
\(H(z, \cdot)\) induces a continuous functional on \(\F^p\) for all \(0<p\le\infty\).
In particular,
by Lemma \ref{lem:reproducing},
\begin{align*}
TK_w(z)=H(z, w).
\end{align*}
Since \(|\alpha|<1/2\), for any \(w\in \C\), by Corollary \ref{cor:nonzero in fp},
\(TK_w\in\F^p\) for all \(0<p\le\infty\).

Combining with Lemma \ref{lem:Gaussian integral}, we obtain
the following lemma.

\begin{lemma}\label{lem:kernels}
For all \(z, w\in\C\),
\[(Tk_w)(z)=H(z, w)\e^{-|w|^2/2}.\]
\[\|Tk_w\|_p=C_{\alpha, p} \mathcal R(w),\qquad 0<p<\infty.\]
Here
\[C_{\alpha, p}=\left(1-4|\alpha|^2\right)^{-1/2p}\exp\left(\frac{1}{2}\frac{|\eta|^2+2\Real(\bar\alpha \eta^2)}{1-4|\alpha|^2}\right).\]
\end{lemma}

Define
\[B(z, w)=\exp\left(-\frac{|w|^2+ |z|^2}{2}+\beta\overline w^{\,2}+\delta\overline w
  +\alpha z^2+\eta z+\gamma z\overline w\right),\quad z, w\in\C.\]
Then
\begin{equation}\label{eq:B}
B(z, w)=\e^{-\frac{|z|^2}{2}}Tk_w(z)=\e^{-\frac{|z|^2+ |w|^2}{2}}H(z, w),\quad z, w\in\C.
\end{equation}

\begin{lemma}\label{lem: B-integral}
For \(0<p<\infty\),
\[
\left(\int_{\C}|B(z, w)|^p\dd\nu_p(z)\right)^{1/p}=C_{\alpha, p}\mathcal R(w),
\]
\[
   \left(\int_{\C}|B(z, w)|^p\dd\nu_p(w)\right)^{1/p}
   =
   C_{\beta,p}\mathcal C(z).
\]
where \(d\nu_p=\frac{p}{2\pi}d A,\) and
\[
C_{\beta, p}=\left(1-4|\beta|^2\right)^{-1/2p}\exp\left(\frac{1}{2}\frac{|\delta|^2+2\Real(\bar\beta \delta^2)}{1-4|\beta|^2}\right).
\]
\end{lemma}
\begin{proof}
By Lemma \ref{lem:kernels},
\[
 \left(\int_{\C}|B(z, w)|^p\dd\nu_p(z)\right)^{1/p}= \left(\int_{\C}|Tk_w(z)\e^{-\frac{|z|^2}{2}}|^p d\nu_p(z)\right)^{1/p}=\|Tk_w\|_p=C_{\alpha, p}\mathcal R(w).
\]
By the correspondence of \(T\) and \(T^\#\), the second equality holds.
\end{proof}

We now proceed to a deeper analysis of \(B(z, w).\)
Write
\[\alpha=\alpha_1+i\alpha_2,\qquad
 \gamma=\gamma_1+i\gamma_2,\qquad
 \eta=\eta_1+i\eta_2,\]
for \(\alpha_1, \alpha_2, \gamma_1, \gamma_2, \eta_1, \eta_2\in\R.\)
Define
\begin{equation*}
 H_\alpha=
 \begin{pmatrix}
  1-2\alpha_1&2\alpha_2\\
  2\alpha_2&1+2\alpha_1
 \end{pmatrix},
\end{equation*}
\[G_\gamma=
 \begin{pmatrix}
  \gamma_1&\gamma_2\\
  -\gamma_2&\gamma_1
 \end{pmatrix},
 \qquad
 e_\eta=\binom{\eta_1}{-\eta_2}.\]
Since \(|\alpha|<1/2\), \(\det H_\alpha=1-4|\alpha|^2>0\). Hence \(H_\alpha\) is positive definite and invertible.
Denote
\begin{equation}\label{eq:L-alpha}
 L_\alpha=H_\alpha^{-1}G_\gamma,
 \qquad
 d_\alpha=H_\alpha^{-1}e_\eta.
\end{equation}
Correspondingly, we can define
\[
H_\beta=
 \begin{pmatrix}
  1-2\beta_1&-2\beta_2\\
  -2\beta_2&1+2\beta_1
 \end{pmatrix},\quad L_\beta=H_\beta^{-1}G_\gamma^{\mathsf T},\quad d_\beta=H_\beta^{-1}e_{\bar\delta},\]
 where \(\beta=\beta_1+i\beta_2\).

\begin{lemma}\label{lem:B-sum}
For \(z, w\in\C\), identify \(z\) and \(w\) with
\[{\rm z}=\binom{\Real z}{\operatorname{Im} z},\qquad {\rm w}=\binom{\Real w}{\operatorname{Im} w}\in\mathbb R^2.\]
Then
\begin{enumerate}
\item[{{\rm (i)}}]
\begin{equation}\label{eq:row}
|B(z, w)|=c_\alpha\mathcal R(w)\exp\bigl(-\frac12({\rm z}-L_\alpha {\rm w}-d_\alpha)^{\mathsf T}
 H_\alpha({\rm z}-L_\alpha {\rm w}-d_\alpha)\bigr),
 \end{equation}
 \begin{equation}\label{eq:row-inequality}
c_\alpha\mathcal R(w) \e^{-\frac{1+2|\alpha|}{2}|{\rm z}-L_\alpha {\rm w}-d_\alpha|^2}\le |B(z, w)|
 \le c_\alpha\mathcal R(w)
 \e^{-\frac{1-2|\alpha|}{2}|{\rm z}-L_\alpha {\rm w}-d_\alpha|^2}.
\end{equation}
where
\begin{equation}\label{eq:row-functions}
 c_\alpha=\exp\left(\frac{1}{2}\frac{|\eta|^2+2\Real(\bar\alpha \eta^2)}{1-4|\alpha|^2}\right).
\end{equation}
\item[{{\rm (ii)}}]
\begin{equation}\label{eq:colum}
 |B(z, w)|
 =c_\beta\mathcal C(z)
 \exp\bigl(-\frac12
 ({\rm w}-L_\beta {\rm z}-d_\beta)^{\mathsf T}
 H_\beta({\rm w}-L_\beta {\rm z}-d_\beta)\bigr),
\end{equation}
\begin{equation*}
c_\beta\mathcal C(z) \e^{-\frac{1+2|\beta|}{2}|{\rm w}-L_\beta {\rm z}-d_\beta|^2}\le |B(z, w)|
\le c_\beta\mathcal C(z)
 \e^{-\frac{1-2|\beta|}{2}|{\rm w}-L_\beta {\rm z}-d_\beta|^2},
\end{equation*}
where
\begin{equation*}
c_{\beta}=\exp\left(\frac{1}{2}\frac{|\delta|^2+2\Real(\bar\beta \delta^2)}{1-4|\beta|^2}\right).
\end{equation*}
\end{enumerate}
\end{lemma}


\begin{proof} (i)\ Note that
\[|B(z, w)|=\exp\bigl(-\frac{|w|^2}{2}+\Real(\beta\bar{w}^2+\delta\bar{w})\bigr)\exp\bigl(-\frac{|z|^2}{2}
 +\Real\!\left(\alpha z^2+(\eta+\gamma\bar w)z\right)\bigr).\]
A direct computation shows that
\[-\frac{|z|^2}{2}
 +\Real\!\left(\alpha z^2+(\eta+\gamma\bar w)z\right)= -\frac12{\rm z}^{\mathsf T}H_\alpha {\rm z}
 +(G_\gamma {\rm w}+e_\eta)^{\mathsf T}{\rm z}.\]

Since \(H_\alpha\) is positive definite, we have
\begin{align*}
(G_\gamma {\rm w}+e_\eta)^{\mathsf T}{\rm z}=&(G_\gamma {\rm w}+e_\eta)^{\mathsf T}(H_\alpha^{-1})^{\mathsf T}H_\alpha {\rm z}\\
=&\!\left(H_\alpha^{-1}(G_\gamma {\rm w}+e_\eta)\right)^{\mathsf T}H_\alpha {\rm z}=(L_\alpha {\rm w}+d_\alpha)^{\mathsf T}H_\alpha {\rm z}.\end{align*}
It follows from
\[(L_\alpha {\rm w}+d_\alpha)^{\mathsf T}H_\alpha {\rm z}=\frac{1}{2}[(L_\alpha {\rm w}+d_\alpha)^{\mathsf T}H_\alpha {\rm z}+{\rm z}^{\mathsf T}H_\alpha(L_\alpha {\rm w}+d_\alpha)]\]
that
\begin{align*}
 &-\frac12{\rm z}^{\mathsf T}H_\alpha {\rm z}+(G_\gamma {\rm w}+e_\eta)^{\mathsf T}{\rm z}\\
 =&-\frac{1}{2}\left({\rm z}^{\mathsf T}H_\alpha {\rm z}-(L_\alpha {\rm w}+d_\alpha)^{\mathsf T}H_\alpha {\rm z}-{\rm z}^{\mathsf T}H_\alpha(L_\alpha {\rm w}+d_\alpha)\right)
 \\
 =&-\frac{1}{2}\left(({\rm z}-L_\alpha {\rm w}-d_\alpha)^{\mathsf T}H_\alpha({\rm z}-L_\alpha {\rm w}-d_\alpha)\right) +\frac{1}{2}(L_\alpha {\rm w}+d_\alpha)^{\mathsf T}H_\alpha(L_\alpha {\rm w}+d_\alpha).
\end{align*}
Therefore,
\[|B(z, w)|=\mathcal B(w)\exp\left(-\frac{1}{2}\left(({\rm z}-L_\alpha {\rm w}-d_\alpha)^{\mathsf T}H_\alpha({\rm z}-L_\alpha {\rm w}-d_\alpha)\right) \right),\]
where
\[\mathcal B(w)=\exp\left(-\frac{|w|^2}{2}+\Real(\beta\bar{w}^2+\delta\bar{w})+\frac{1}{2}(L_\alpha {\rm w}+d_\alpha)^{\mathsf T}H_\alpha(L_\alpha {\rm w}+d_\alpha)\right).\]
Since
\[\frac1\pi\int_{\mathbb R^2}
 \exp\left\{-\frac12{\rm z}^{\mathsf T}H_\alpha {\rm z}\right\}dA(z)
 =\frac{2}{\sqrt{1-4|\alpha|^2}},\]
 we obtain
\[\frac1\pi\int_{\C}|B(z, w)|dA(z)=\frac{2}{\sqrt{1-4|\alpha|^2}}\mathcal B(w).\]
On the other hand, by Lemma \ref{lem: B-integral},
\[\frac1\pi\int_{\C}|B(z, w)|dA(z)=2C_{\alpha, 1}\mathcal R(w).\]
Thus, we have
\[\mathcal B(w)=c_\alpha\mathcal R(w),\]
where
\(
 c_\alpha=\sqrt{1-4|\alpha|^2}C_{\alpha, 1}.
\) This gives \eqref{eq:row} and \eqref{eq:row-functions}.

Since \(H_\alpha\) is positive definite  and
\(1+2|\alpha|, 1-2|\alpha|\) are the  two eigenvalues of \(H_\alpha\), \eqref{eq:row-inequality} follows from \eqref{eq:row}.


(ii) is obtained as (i) by the correspondence of \(T\) and \(T^\#\).
\end{proof}

For simplicity, we shall not distinguish between \(z\in\C\) and \(\rm z=\binom{\Real z}{\operatorname{Im} z}\in\R^2\)
in what follows.

\begin{lemma}\label{lem: B-infty}
For any \(w\in\C\),
\[\|Tk_w\|_\infty=\sup_{z\in\C}|B(z, w)|=c_\alpha\mathcal{R}(w),\]
\end{lemma}
\begin{proof} Since
\[\|Tk_w\|_\infty=\sup_{z\in\C}|Tk_w(z)\e^{-|z|^2/2}|=\sup_{z\in\C}|B(z, w)|\]
and
by \eqref{eq:row},
\[c_\alpha\mathcal{R}(w)=|B(L_\alpha w+d_\alpha, w)|\le \sup_{z\in\C}|B(z, w)|\le c_\alpha\mathcal{R}(w),\]
we get
\[\|Tk_w\|_\infty=\sup_{z\in\C}|B(z, w)|=c_\alpha\mathcal{R}(w).\qedhere\]
\end{proof}

\begin{lemma}\label{lem:R-C-equivalence}
\(\sup_{w\in\C}\mathcal R(w)<\infty\) if and only if \(\sup_{z\in\C}\mathcal C(z)<\infty.\)
Moreover, assume \(\gamma\neq 0\), then
\(\mathcal R(w)\rightarrow 0\) as \(|w|\to\infty\) if and only if
\(\mathcal C(z)\rightarrow 0\) as \(|z|\to\infty\).
\end{lemma}

\begin{proof}
By Lemma \ref{lem: B-infty},
\[
c_\alpha\sup_{w\in\C}\mathcal{R}(w)=\sup_{w\in\C}\sup_{z\in\C}|B(z, w)|.
\]
It follows from  \eqref{eq:colum} that for any \(z\in\C\),
\[c_\beta\mathcal{C}(z)=|B(z, L_\beta z+d_\beta)|\le \sup_{w\in\C}|B(z, w)|\le c_\beta\mathcal{C}(z).\]
Hence
\[
c_\beta\sup_{z\in\C}\mathcal{C}(z)=\sup_{z\in\C}\sup_{w\in\C}|B(z, w)|.
\]
Therefore,
\[
c_\alpha\sup_{w\in\C}\mathcal{R}(w)=c_\beta\sup_{z\in\C}\mathcal{C}(z).
\]
The first assertion follows.

Assume \(\gamma\neq 0\), then \(\det G_\gamma=|\gamma|^2>0\), hence \(G_\gamma\) is invertible and thus \(L_\alpha\) is invertible by \eqref{eq:L-alpha}.
 By Lemma \ref{lem:B-sum},
\begin{align}\label{eq:B-z-w}
&\ c_\alpha\mathcal R(w)
 \exp\bigl(-\frac12
 (z-L_\alpha w-d_\alpha)^{\mathsf T}
 H_\alpha(z-L_\alpha w-d_\alpha)\bigr)\nonumber\\
 =&\ c_\beta\mathcal C(z)
 \exp\bigl(-\frac12
 (w-L_\beta z-d_\beta)^{\mathsf T}
 H_\beta(w-L_\beta z-d_\beta)\bigr).
\end{align}

Assume \(\mathcal C(z)\rightarrow 0\) as \(|z|\rightarrow\infty\). For any \(w\neq 0\), let \(z=L_\alpha w+d_\alpha\), then when
\(|w|\rightarrow \infty\), we must have \(|{\rm z}|\rightarrow \infty\)  since \(L_\alpha\) is invertible. So by \eqref{eq:B-z-w}
\[c_\alpha \mathcal{R}(w)\le c_\beta\mathcal C(L_\alpha w+d_\alpha)\rightarrow 0 \qquad \text{as}\quad |w|\rightarrow\infty.\]
The converse is similar.
\end{proof}

\begin{remark}
When \(\gamma=0\), for \(z, w\in\C\),
\[\Theta_\alpha(w)=-\frac{1}{2}|w|^2+\Real(\bar{\beta}w^2+\bar{\delta}w),\qquad \Theta_\beta(z)=-\frac{1}{2}|z|^2+\Real(\alpha z^2+\eta z).\]
Since \(|\alpha|<1/2\), \(|\beta|<1/2\), naturally, we have
\[
 \Theta_\alpha(w)\rightarrow-\infty\quad(|w|\to\infty),
 \qquad
 \Theta_\beta(z)\rightarrow-\infty\quad(|z|\to\infty).
\]
Hence \(\mathcal R(w)\rightarrow 0\) as \(|w|\to\infty\) and
\(\mathcal C(z)\rightarrow 0\) as \(|z|\to\infty\).
\end{remark}

\section{Passage from the integral operator  to the matrix operator}\label{sec:passage}

Let \(\Lambda=r\mathbb{Z}^2\), where \(0<r<r_0\) as in the statement of Lemma \ref{lem:dense}. Now we introduce the matrix operator
\[
 (A\zeta)_\mu=\sum_{\lambda\in\Lambda}B(\mu,\lambda)\zeta_\lambda,
 \qquad  \zeta=\{\zeta_\lambda\}_{\lambda\in\Lambda},\quad \mu\in\Lambda.
\]
This section is devoted to establishing the relationship between the integral operator \(T\) and the matrix operator \(A\), as well as to a detailed investigation of \(A\). Assume throughout that
\(|\alpha|<\frac{1}{2}, |\beta|<\frac{1}{2}\).

As in the Remark \ref{rem:SpRp}, let \(R_q: \F^q\to \ell^q(\Lambda)\) and \(S_p: l^p(\Lambda)\to \F^p\) be the bounded operators from
Lemma \ref{lem:dense}. For any finitely supported sequence \(\zeta\) on \(\Lambda\) and any \(\mu\in\Lambda\), by \eqref{eq:B},
\[
(R_qTS_p\zeta)_{\mu}=(TS_p\zeta)(\mu)e^{-\frac{|\mu|^2}{2}}=
\sum_{\lambda\in \Lambda}\zeta_\lambda (Tk_\lambda)(\mu)e^{-\frac{|\mu|^2}{2}}=\sum_{\lambda\in \Lambda}B(\mu, \lambda)\zeta_\lambda.
\]
That is,
\begin{align}\label{eq:RTS=A}
 R_qTS_p\zeta=A\zeta.
\end{align}

\begin{proposition}\label{prop:passage}
Let \(0<p<\infty, 0<q\le\infty\), and let \(A=(B(\mu,\lambda))_{\mu, \lambda\in\Lambda}\) be the lattice
matrix defined above. Then \(T:\F^p\to\F^q\) is bounded if and only if \(A: \ell^p(\Lambda)\to\ell^q(\Lambda)\) is bounded.
\end{proposition}
\begin{proof}

Since the set of finitely supported sequences are dense in \(\ell^p(\Lambda)\),  the boundedness of \(A: \ell^p(\Lambda)\to\ell^q(\Lambda)\) follows from \eqref{eq:RTS=A}, the boundedness of \(T:\F^p\to\F^q\), \(R_q\) and \(S_p\).

Suppose that \(A:\ell^p(\Lambda)\to\ell^q(\Lambda)\) is bounded. For any finitely supported sequence \(\zeta\) on \(\Lambda\), by Lemma \ref{lem:dense} (ii) and \eqref{eq:RTS=A},
\begin{align}\label{eq:A-T}
\|TS_p\zeta\|_q\le C_{r, q}\|R_q(TS_p\zeta)\|_{\ell^q(\Lambda)}=C_{r, q}\|A\zeta\|_{\ell^q(\Lambda)}\le C_{r, q}\|A\|\|\zeta\|_{\ell^p(\Lambda)},
\end{align}
Given \(f\in\F^p\),
choose an atomic representation \(f=S_p\zeta\) with
\(\|\zeta\|_{\ell^p(\Lambda)}\le C_{r, p}\|f\|_p\). Truncating \(\zeta\) gives finitely supported sequences
\(\zeta^{(n)}\to \zeta\) in \(\ell^p(\Lambda)\).
Applying \eqref{eq:A-T} to \(\zeta^{(n)}-\zeta^{(m)}\)
shows that \(\{TS_p\zeta^{(n)}\}\) is a Cauchy sequence in \(\F^q\), hence \(\lim_{n\to\infty}TS_p\zeta^{(n)}\) exists in \(\F^q\). On the other hand, since \(S_p\zeta^{(n)}\to S_p\zeta\) in \(\F^p\),  by Corollary \ref{cor:functional}, for any \(z\in \C\),
\[\lim_{n\to\infty}(TS_p\zeta^{(n)})(z)= (TS_p\zeta)(z)=(Tf)(z)\]
It follows from \eqref{eq:A-T} that
\[\|Tf\|_q\le C_{r, q}\|A\|\|\zeta\|_{\ell^p(\Lambda)}\le C_{r, q}C_{r, p}\|A\|\|f\|_p.\]
Thus \(T:\F^p\to\F^q\) is bounded
\end{proof}

In the following, we always assume that \(\gamma\neq0\). Now we use the notation as in Lemma \ref{lem:B-sum}.
Since
\(L_\alpha=H_\alpha^{-1}G_\gamma\)
and
\[
  \det G_\gamma=|\gamma|^2>0,
  \qquad
  \det H_\alpha=1-4|\alpha|^2>0,
\]
we have
\[
  \det L_\alpha
  =\frac{|\gamma|^2}{1-4|\alpha|^2}>0.
\]
Thus \(L_\alpha\) is invertible.
Let \(m_\alpha>0\) denote the smallest singular value of
\(L_\alpha\). Then
\begin{align}\label{eq:Lalpha}
|L_\alpha z-L_\alpha w|\ge m_\alpha|z-w|,\qquad z, w\in \C.
\end{align}
For \(\lambda\in\Lambda\), let
\[z_\lambda=L_\alpha\lambda+d_\alpha.\]
By Lemma \ref{lem: lattices} and \eqref{eq:Lalpha},  we have
\begin{equation}\label{eq:sup1}
\begin{array}{c}
\displaystyle{ \sup_{\lambda\in\Lambda}\sum_{\mu\in\Lambda}\e^{-\frac{1-2|\alpha|}{2}|\mu-z_\lambda|^2}<\infty,}\\
\displaystyle{
\sup_{\mu\in\Lambda}\sum_{\lambda\in\Lambda}\e^{-\frac{1-2|\alpha|}{2}|\mu-z_\lambda|^2}\le  \sup_{\mu\in\Lambda}\sum_{\lambda\in\Lambda}\e^{-\frac{1-2|\alpha|}{2} m_\alpha^2|L_\alpha^{-1}(\mu-d_\alpha)-\lambda|^2}<\infty.}
\end{array}
\end{equation}

Denote
\[ r_\lambda=\mathcal R(\lambda),\qquad \lambda\in\Lambda\]
 and define the diagonal matrix \(D\) as
 \[D\zeta=\{r_\lambda \zeta_\lambda\}_{\lambda\in\Lambda},\qquad \zeta=\{\zeta_\lambda\}_{\lambda\in\Lambda}.\]

\begin{lemma}\label{lem:AlessD}
Assume that \(\gamma\neq0\). Fix \(0<q\le\infty\).
There exists a positive  constant \(C\) (depending only on  the fixed parameters in \(H\), the lattice and \(q\)) such that for any finitely supported sequence \(\zeta=\{\zeta_\lambda\}_{\lambda\in\Lambda}\),
\begin{equation*}
  \|A\zeta\|_{\ell^q(\Lambda)}
  \le C
  \|D\zeta\|_{\ell^q(\Lambda)}.
\end{equation*}
\end{lemma}
\begin{proof}
Let \(\zeta=\{\zeta_\lambda\}\) be finitely supported on \(\Lambda\). Then by \eqref{eq:row-inequality},
\begin{align}\label{eq:Amu-bound}
 |(A\zeta)_\mu|\le c_\alpha\sum_{\lambda\in\Lambda}\e^{-\frac{1-2|\alpha|}{2} |\mu-z_\lambda|^2}|r_\lambda\zeta_\lambda|,\qquad \mu\in\Lambda.
\end{align}

Suppose first that \(1\le q<\infty\).  By \eqref{eq:Amu-bound} and the weighted H\"older inequality,
\begin{align*}
 |(A\zeta)_\mu|^q
 &\le
c_\alpha^q\left(\sum_{\lambda\in\Lambda}\e^{-\frac{1-2|\alpha|}{2} |\mu-z_\lambda|^2}\right)^{q-1}\sum_{\lambda\in\Lambda}
 \e^{-\frac{1-2|\alpha|}{2} |\mu-z_\lambda|^2}(|r_\lambda\zeta_\lambda|)^q\\
 &\le
 c_\alpha^q\left(\sup_{\mu\in\Lambda}\sum_{\lambda\in\Lambda}\e^{-\frac{1-2|\alpha|}{2} |\mu-z_\lambda|^2}\right)^{q-1}\sum_{\lambda\in\Lambda}
 \e^{-\frac{1-2|\alpha|}{2} |\mu-z_\lambda|^2}(|r_\lambda\zeta_\lambda|)^q.
\end{align*}
Therefore, by \eqref{eq:sup1},
\begin{align*}\|A\zeta\|_{\ell^q(\Lambda)}^q &\le
c_\alpha^q\left(\sup_{\mu\in\Lambda}\sum_{\lambda\in\Lambda}\e^{-\frac{1-2|\alpha|}{2} |\mu-z_\lambda|^2}\right)^{q-1}
\sum_{\lambda\in\Lambda}
 \big{(}\sum_{\mu\in\Lambda}\e^{-\frac{1-2|\alpha|}{2} |\mu-z_\lambda|^2}\big{)}|r_\lambda\zeta_\lambda|^q\\
&\le
c_\alpha^q\left(\sup_{\mu\in\Lambda}\sum_{\lambda\in\Lambda}\e^{-\frac{1-2|\alpha|}{2} |\mu-z_\lambda|^2}\right)^{q-1}
 \left(\sup_{\lambda\in\Lambda}\sum_{\mu\in\Lambda}\e^{-\frac{1-2|\alpha|}{2} |\mu-z_\lambda|^2}\right)\sum_{\lambda\in\Lambda}
|r_\lambda\zeta_\lambda|^q\\
&\le
 C\|D\zeta\|_{\ell^q(\Lambda)}^q,
\end{align*}
for some \(C>0\) .

For \(0<q\le1\), by \eqref{eq:Amu-bound},
\begin{align*}
 |(A\zeta)_\mu|^q \le c_\alpha^q
  \sum_{\lambda\in\Lambda}
 \e^{-\frac{1-2|\alpha|}{2} q|\mu-z_\lambda|^2}|r_\lambda\zeta_\lambda|^q,
\end{align*}
and again by \eqref{eq:sup1},
\begin{align*}
 \|A\zeta\|_{\ell^q(\Lambda)}^q &\le c_\alpha^q\left(\sup_{\lambda\in\Lambda}\sum_{\mu\in\Lambda}\e^{-\frac{1-2|\alpha|}{2} q|\mu-z_\lambda|^2}\right)
 \sum_{\lambda\in\Lambda}
|r_\lambda\zeta_\lambda|^q \le C\|D\zeta\|_{\ell^q(\Lambda)}^q.
\end{align*}
for some \(C>0\).

For \(q=\infty\),  by \eqref{eq:Amu-bound},
\begin{align*}
 \|A\zeta\|_{\ell^\infty(\Lambda)} &\le
 c_\alpha\sup_{\mu\in\Lambda}\sum_{\lambda\in\Lambda}\e^{-\frac{1-2|\alpha|}{2} |\mu-z_\lambda|^2}|r_\lambda\zeta_\lambda|\le C\|D\zeta\|_{\ell^\infty(\Lambda)}
\end{align*}
for some \(C>0\).
\end{proof}

\begin{lemma}\label{lem:DlessA}
Let \(0<q\le\infty\). There exists a sublattice decomposition of \(\Lambda=r\mathbb{Z}^2\) as in {\rm Lemma \ref{lem: lattices decomposition} (ii)}, and a positive constant \(c\) (depending only on  the fixed parameters in \(H(z, w)\), the lattice and \(q\)) with the following properties. For each \(1 \leq j \leq N\), there is an injective map
\[
  \pi_j:\Lambda_j\longrightarrow\Lambda,
\]
and the corresponding coordinate projection
\(P_j:\ell^q(\Lambda)\to\ell^q(\pi_j(\Lambda_j))\) such that for every finitely supported sequence \(\zeta\) on \(\Lambda_j\),
\begin{equation*}
 c
  \|D\zeta\|_{\ell^q(\Lambda_j)}
  \le \|P_jA\zeta\|_{\ell^q(\pi_j(\Lambda_j))}.
\end{equation*}
\end{lemma}
\begin{proof}
Take a decomposition of \(\Lambda\) as in Lemma \ref{lem: lattices decomposition} for  some \(K\) to be determined later. Fix \(1\le j\le N\). For each
\(\lambda\in\Lambda_j\), choose a nearest lattice point
\(\pi_j(\lambda)\in\Lambda\) to \(z_\lambda\).
Then by \eqref{eq:nearest-lattice-point},
\[
|\pi_j(\lambda)-z_\lambda|\le r/\sqrt{2}
\]
and for \(\lambda', \lambda''\in \Lambda_j\), by \eqref{eq:Lalpha}
\[\begin{aligned}
|\pi_j(\lambda')-\pi_j(\lambda'')|&\ge |z_{\lambda'}-z_{\lambda''}|-|\pi_j(\lambda')-z_{\lambda'}|-|\pi_j(\lambda'')-z_{\lambda''}|\\
&\ge m_\alpha|\lambda'-\lambda^{\prime\prime}|-\sqrt{2}r.
\end{aligned}\]
By Lemma \ref{lem: lattices decomposition}, \(|\lambda'-\lambda^{\prime\prime}|\ge K\), so \(\pi_j: \Lambda_j\to \Lambda\) is injective when \(K\) is sufficiently
large.

By \eqref{eq:row-inequality},
\begin{align}\label{eq:central-lower}
|a_{\pi_j(\lambda),\lambda}|=|B(\pi_j(\lambda),\lambda)|
&\ge  c_\alpha r_\lambda \exp\bigl(-\frac{1+2|\alpha|}{2}|\pi_j(\lambda)-z_\lambda|^2\bigr)\nonumber\\
&\ge c_\alpha r_\lambda \exp\bigl(-\frac{1+2|\alpha|}{4}r^2\bigr)=c_{\alpha ,r}r_\lambda,
\end{align}
where \(c_{\alpha ,r}=c_\alpha\e^{-(1+2|\alpha|)r^2/4}\),
and by \eqref{eq:Lalpha},
\begin{align}\label{eq:central-upper}
|a_{\pi_j(\lambda),\lambda'}|&\le c_\alpha r_{\lambda'}\e^{-\frac{1-2|\alpha|}{2}|\pi_j(\lambda)-z_{\lambda'}|^2}\le c_\alpha r_{\lambda'}\e^{-\frac{1-2|\alpha|}{2}[\frac{1}{2}|z_\lambda-z_{\lambda'}|^2-|\pi_j(\lambda)-z_\lambda|^2]}\nonumber\\
&\le c_\alpha\e^{\frac{(1-2|\alpha|)r^2}{4}}r_{\lambda'}\e^{-\frac{(1-2|\alpha|)}{4}|z_\lambda-z_{\lambda'}|^2}\le c_\alpha\e^{\frac{(1-2|\alpha|)r^2}{4}} r_{\lambda'}\e^{-\frac{(1-2|\alpha|)}{4}m_\alpha^2|\lambda-\lambda'|^2}\nonumber\\
&=C_{\alpha, r} r_{\lambda'}\e^{-\frac{(1-2|\alpha|)}{4}m_\alpha^2|\lambda-\lambda'|^2},
\end{align}
where \(C_{\alpha, r}=c_\alpha\e^{(1-2|\alpha|)r^2/4}.\)

For \(\zeta=\{\zeta_\lambda\}\) finitely supported on \(\Lambda_j\), let
\[P_j A\zeta=\{(A\zeta)_{\pi_j(\lambda)}\}_{\lambda\in\Lambda_j}=\{\sum_{\lambda'\in\Lambda_j}a_{\pi_j(\lambda), \lambda'}\zeta_{\lambda'}\}_{\lambda\in\Lambda_j}\]
and write
\[P_jA\zeta=Q\zeta+E\zeta,\]
where
\[(Q\zeta)_{\pi_j(\lambda)}=a_{\pi_j(\lambda),\lambda}\zeta_\lambda\quad \text{and}
\quad   (E\zeta)_{\pi_j(\lambda)}=\sum_{\lambda'\in\Lambda_j\, \lambda'\neq\lambda}
  a_{\pi_j(\lambda),\lambda'}\zeta_{\lambda'}.\]

 \eqref{eq:central-lower} gives
\begin{align*}
  \|Q\zeta\|_{\ell^q(\pi_j(\Lambda_j))}^q&=\sum_{\lambda\in\Lambda_j}|(Q\zeta)_{\pi_j(\lambda)}|^q
  =\sum_{\lambda\in\Lambda_j}|a_{\pi_j(\lambda),\lambda}\zeta_\lambda|^q\\
  &\ge
  c_{\alpha, r}^q\sum_{\lambda\in\Lambda_j}|r_\lambda \zeta_\lambda|^q = c_{\alpha, r}^q\|D\zeta\|_{\ell^q(\Lambda_j)}^q,\quad 0<q<\infty,
\end{align*}
\begin{align*}
  \|Q\zeta\|_{\ell^\infty(\pi_j(\Lambda_j))}&=\sup_{\lambda\in\Lambda_j}|(Q\zeta)_{\pi_j(\lambda)}|
  =\sup_{\lambda\in\Lambda_j}|a_{\pi_j(\lambda),\lambda}\zeta_\lambda|\\
  &\ge
  c_{\alpha, r}\sup_{\lambda\in\Lambda_j}|r_\lambda \zeta_\lambda| = c_{\alpha, r}\|D\zeta\|_{\ell^\infty(\Lambda_j)},
\end{align*}
and \eqref{eq:central-upper} and Lemma \ref{lem: lattices decomposition} (ii) give
\begin{align*}
\|E\zeta\|_{\ell^q(\pi_j(\Lambda_j))}^q= &\sum_{\lambda_\in\Lambda_j}|(E\zeta)_{\pi_j(\lambda)}|^q
  \le \sum_{\lambda_\in\Lambda_j}\bigg(
  \sum_{\lambda'\in\Lambda_j, \lambda'\neq\lambda}
  |a_{\pi_j(\lambda),\lambda'}\zeta_{\lambda'}|\bigg)^q\\
  \le & C_{\alpha, r}^q\sum_{\lambda_\in\Lambda_j}\bigg(
  \sum_{\lambda'\in\Lambda_j, \lambda'\neq\lambda}
  \e^{-\frac{(1-2|\alpha|)}{4}m_\alpha^2|\lambda-\lambda'|^2}|r_{\lambda'}\zeta_{\lambda'}|\bigg)^q,\quad 0<q<\infty,
\end{align*}
\begin{align*}
\|E\zeta\|_{\ell^\infty(\pi_j(\Lambda_j))}= &\sup_{\lambda_\in\Lambda_j}|(E\zeta)_{\pi_j(\lambda)}|\le
   \sup_{\lambda_\in\Lambda_j}
  \sum_{\lambda'\in\Lambda_j, \lambda'\neq\lambda}
  |a_{\pi_j(\lambda),\lambda'}\zeta_{\lambda'}|\\
  \le & C_{\alpha, r}\sup_{\lambda_\in\Lambda_j}\left(
  \sum_{\lambda'\in\Lambda_j, \lambda'\neq\lambda}
  \e^{-\frac{(1-2|\alpha|)}{4}m_\alpha^2|\lambda-\lambda'|^2}|r_{\lambda'}\zeta_{\lambda'}|\right)\\
  \le& C_{\alpha, r}\left(\sup_{\lambda_\in\Lambda_j}
  \sum_{\lambda'\in\Lambda_j, \lambda'\neq\lambda}
  \e^{-\frac{(1-2|\alpha|)}{4}m_\alpha^2|\lambda-\lambda'|^2}\right)
  \sup_{\lambda'\in\Lambda_j}|r_{\lambda'}\zeta_{\lambda'}|\\
  \le & C_{\alpha, r}C_K\|D\zeta\|_{\ell^\infty(\Lambda_j)}.
  \end{align*}

When \(1\le q<\infty\),
 \begin{align*}
& \bigg(
  \sum_{\lambda'\in\Lambda_j, \lambda'\neq\lambda}
  \e^{-\frac{(1-2|\alpha|)}{4}m_\alpha^2|\lambda-\lambda'|^2}|r_{\lambda'}\zeta_{\lambda'}|\bigg)^q\\
\le&\ \bigg(
  \sum_{\lambda'\in\Lambda_j, \lambda'\neq\lambda}
  \e^{-\frac{(1-2|\alpha|)}{4}m_\alpha^2|\lambda-\lambda'|^2}\bigg)^{q-1} \sum_{\lambda'\in\Lambda_j, \lambda'\neq\lambda} \e^{-\frac{(1-2|\alpha|)}{4}m_\alpha^2|\lambda-\lambda'|^2}|r_{\lambda'}\zeta_{\lambda'}|^q \\
\le&\ \bigg(\sup_{\lambda\in\Lambda_j}
  \sum_{\lambda'\in\Lambda_j, \lambda'\neq\lambda}
  \e^{-\frac{(1-2|\alpha|)}{4}m_\alpha^2|\lambda-\lambda'|^2}\bigg)^{q-1}\sum_{\lambda'\in\Lambda_j, \lambda'\neq\lambda} \e^{-\frac{(1-2|\alpha|)}{4}m_\alpha^2|\lambda-\lambda'|^2}|r_{\lambda'}\zeta_{\lambda'}|^q.
 \end{align*}
Thus by Lemma \ref{lem: lattices decomposition} (ii),
 \begin{align*}
 &\|E\zeta\|_{\ell^q(\pi_j(\Lambda_j))}^q\\
 \le & C_{\alpha, r}^q\bigg(\sup_{\lambda\in\Lambda_j}
  \sum_{\lambda'\in\Lambda_j, \lambda'\neq\lambda}
  \e^{-\frac{(1-2|\alpha|)}{4}m_\alpha^2|\lambda-\lambda'|^2}\bigg)^{q-1}
  \sum_{\lambda_\in\Lambda_j}\sum_{\lambda'\in\Lambda_j, \lambda'\neq\lambda} \e^{-\frac{(1-2|\alpha|)}{4}m_\alpha^2|\lambda-\lambda'|^2}|r_{\lambda'}\zeta_{\lambda'}|^q\\
  \le& C_{\alpha, r}^q\left(
  \sup_{\lambda\in\Lambda_j}\sum_{\lambda'\in\Lambda_j, \lambda'\neq\lambda}\e^{-\frac{(1-2|\alpha|)}{4}m_\alpha^2|\lambda-\lambda'|^2}\right)^q
   \sum_{\lambda'\in\Lambda_j} |r_{\lambda'}\zeta_{\lambda'}|^q\\
   \le & C_{\alpha, r}^qC_K^q\|D\zeta\|_{\ell^q(\Lambda_j)}^q.
  \end{align*}
The triangle inequality yields
\[
  \|P_jA\zeta\|_{\ell^q(\pi_j(\Lambda_j))}
  \ge \|Q\zeta\|_{\ell^q(\pi_j(\Lambda_j))}-\|E\zeta\|_{\ell^q(\pi_j(\Lambda_j))}
  \ge   (c_{\alpha, r}-C_{\alpha, r}C_{K})
  \|D\zeta\|_{\ell^q(\Lambda_j)}.
\]

When \(0< q<1\),
 \[\begin{aligned}
 \bigg(
  \sum_{\lambda'\in\Lambda_j, \lambda'\neq\lambda}
  \e^{-\frac{(1-2|\alpha|)}{4}m_\alpha^2|\lambda-\lambda'|^2}|r_{\lambda'}\zeta_{\lambda'}|\bigg)^q
  &\le
  \sum_{\lambda'\in\Lambda_j, \lambda'\neq\lambda}
  \e^{-\frac{(1-2|\alpha|)}{4}qm_\alpha^2|\lambda-\lambda'|^2}|r_{\lambda'}\zeta_{\lambda'}|^q.
 \end{aligned}\]
Thus Lemma \ref{lem: lattices decomposition} (ii) gives
\begin{align*}
 \|E\zeta\|_{\ell^q(\pi_j(\Lambda_j))}^q\le & C_{\alpha, r}^q\sum_{\lambda\in\Lambda_j}\sum_{\lambda'\in\Lambda_j, \lambda'\neq\lambda} \e^{-\frac{(1-2|\alpha|)}{4}qm_\alpha^2|\lambda-\lambda'|^2}|r_{\lambda'}\zeta_{\lambda'}|^q\\
  \le&C_{\alpha, r}^q\left(
  \sup_{\lambda'\in\Lambda_j}\sum_{\lambda\in\Lambda_j, \lambda\neq\lambda'}\e^{-\frac{(1-2|\alpha|)}{4}qm_\alpha^2|\lambda-\lambda'|^2}\right)
   \sum_{\lambda'\in\Lambda_j} |r_{\lambda'}\zeta_{\lambda'}|^q\\
   \le & C_{\alpha, r}^qC_K\|D\zeta\|_{\ell^q(\Lambda_j)}^q.
  \end{align*}
Then
\[
  \|Q\zeta\|_{\ell^q(\pi_j(\Lambda_j))}^q
  \le
  \|P_jA\zeta\|_{\ell^q(\pi_j(\Lambda_j))}^q+\|E\zeta\|_{\ell^q(\pi_j(\Lambda_j))}^q,
\]
and therefore
\[
  \|P_jA\zeta\|_{\ell^q(\pi_j(\Lambda_j))}^q
  \ge
   (c_{\alpha, r}^q-C_{\alpha, r}^qC_{K})
  \|D\zeta\|_{\ell^q(\Lambda_j)}^q.
\]

When \(q=\infty\), the triangle inequality yields
\[
  \|P_jA\zeta\|_{\ell^\infty(\pi_j(\Lambda_j))}
  \ge \|Q\zeta\|_{\ell^\infty(\pi_j(\Lambda_j))}-\|E\zeta\|_{\ell^\infty(\pi_j(\Lambda_j))}\ge
   (c_{\alpha, r}-C_{\alpha, r}C_{K})
  \|D\zeta\|_{\ell^\infty(\Lambda_j)}.
\]

By Lemma \ref{lem: lattices decomposition}, \(C_K\to 0\) as \(K\to\infty\). So, we can take \(K\) sufficiently large such that
\(c_{\alpha, r}-C_{\alpha, r}C_{K}>0\) or \(c_{\alpha, r}^q-C_{\alpha, r}^qC_{K}>0\) when \(0<q<1\). Therefore, for \(0<q\le\infty\), there exists a positive constant \(c\) such that for any finitely supported sequence
\(\zeta\) on \(\Lambda_j\), \(1\le j\le N\),
\begin{equation*}
 c
  \|D\zeta\|_{\ell^q(\Lambda_j)}
  \le \|P_jA\zeta\|_{\ell^q(\pi_j(\Lambda_j))}.\qedhere
\end{equation*}
\end{proof}

 For fixed \(\mu\in\Lambda\), by Lemma \ref{lem:quadratic-integrability}, \(B(\mu, \lambda)\in\ell^t(\Lambda)(\subset \ell^\infty(\Lambda))\) for any \(0<t<\infty\) since \(|\beta|<1/2\). Therefore, for any \(0<p\le\infty\) and \(\zeta\in\ell^p(\Lambda)\), \(\sum_{\lambda\in\Lambda}B(\mu,\lambda)\zeta_\lambda\) is absolutely convergent and \((A\zeta)_\mu\) is well-defined. In particular, let \(\{\zeta^{(n)}\}\) be the truncated sequence of \(\zeta\), we have
\begin{eqnarray}\label{eq:A-truncation}
\begin{split}
\bigl(\sum_{\mu\in\Lambda}|(A\zeta)_\mu|^q\bigr)^{1/q}\le \limsup_{n}\|A\zeta^{(n)}\|_{\ell^q(\Lambda)},\qquad 0<q<\infty\\
\sup_{\mu\in\Lambda}|(A\zeta)_\mu|\le \limsup_{n}\|A\zeta^{(n)}\|_{\ell^\infty(\Lambda)}.
\end{split}
\end{eqnarray}
By the definition of \(D\), we also have
\begin{eqnarray}\label{eq:D-truncation}
\begin{split}
\bigl(\sum_{\mu\in\Lambda}|(D\zeta)_\mu|^q\bigr)^{1/q}\le \limsup_{n}\|D\zeta^{(n)}\|_{\ell^q(\Lambda)},\qquad 0<q<\infty\\
\sup_{\mu\in\Lambda}|(D\zeta)_\mu|\le \limsup_{n}\|D\zeta^{(n)}\|_{\ell^\infty(\Lambda)}.
\end{split}
\end{eqnarray}

\begin{proposition}\label{prop:A=D}
Assume that \(\gamma\neq0\). Fix \(0<p, q\le\infty\). Then
\(A: \ell^{p}(\Lambda)\to \ell^q(\Lambda)\) is bounded if and only if \(D: \ell^{p}(\Lambda)\to \ell^q(\Lambda)\) is bounded.
\end{proposition}
\begin{proof}
Assume that \(D: \ell^{p}(\Lambda)\to \ell^q(\Lambda)\) is bounded.
For \(\zeta\in\ell^p(\Lambda)\), let  \(\{\zeta^{(n)}\}\) be the truncated sequence of \(\zeta\).
By Lemma \ref{lem:AlessD}, there exists a positive constant \(C\) such that
\[
 \|A\zeta^{(n)}\|_{\ell^q(\Lambda)}
  \le C
  \|D\zeta^{(n)}\|_{\ell^q(\Lambda)}\le C  \|D\zeta\|_{\ell^q(\Lambda)}.
\]
By \eqref{eq:A-truncation}, we have
\[
 \|A\zeta\|_{\ell^q(\Lambda)}\le C  \|D\zeta\|_{\ell^q(\Lambda)}.
\]
Hence \(A: \ell^{p}(\Lambda)\to \ell^q(\Lambda)\) is bounded.

Assume that \(A: \ell^{p}(\Lambda)\to \ell^q(\Lambda)\) is bounded. For any finitely supported sequence \(\zeta\) on \(\Lambda\), by Lemma \ref{lem:DlessA}, there exists a positive constant \(c\) such that for
\begin{align*}
c^q\|D\zeta\|_{\ell^q(\Lambda)}^q=c^q\sum_{j=1}^N\|D\zeta\|_{\ell^q(\Lambda_j)}^q\le \sum_{j=1}^N\|P_j(A\zeta)\|_{\ell^q(\pi_j(\Lambda_j))}^q
\le N\|A\zeta\|_{\ell^q(\Lambda)}^q,\quad 0<q<\infty;
\end{align*}
\begin{align*}
c\|D\zeta\|_{\ell^\infty(\Lambda)}=c\max_{1\le j\le N}\|D\zeta\|_{\ell^\infty(\Lambda_j)}\le \max_{1\le j\le N}\|P_j(A\zeta)\|_{\ell^\infty(\pi_j(\Lambda_j))}
\le \|A\zeta\|_{\ell^\infty(\Lambda)}.
\end{align*}
Combining with \eqref{eq:D-truncation}, the boundedness of  \(D: \ell^{p}(\Lambda)\to \ell^q(\Lambda)\) follows.
\end{proof}

\section{Boundedness and Compactness}

In this section, we give the proof of the main results.

\subsection{The case of \(\gamma\neq 0\)}

\begin{lemma}\label{lem:necessary}
Assume \(\gamma\neq 0\).
If \(T: \F^p\rightarrow \F^q\) is bounded for \(0<p, q\le \infty\), then
\[|\alpha|<1/2,\qquad |\beta|<1/2.\]
\end{lemma}
\begin{proof}
Since \(T: \F^p\rightarrow \F^q\) is bounded, for any \(f\in\F^p\),  \(Tf\in\F^q\). By Lemma \ref{lem:reproducing}, for any \(z\in\C\), we have
\[|\langle f, h_z\rangle|=|(Tf)(z)|\le e^{\frac{|z|^2}{2}}\|Tf\|_q\le e^{\frac{|z|^2}{2}}\|T\|\|f\|_p,\]
where \(h_z(w)=\overline{H(z, w)}.\)
Hence \(h_z\) induces a continuous functional on \(\F^p\).
When \(0<p<\infty\), by Lemma \ref{lem:duality}  and Lemma \ref{lem:inclusion},
\(h_z\in\F^{p'}\subset \F^\infty\);
when \(p=\infty\), by Lemma \ref{lem:inclusion}, \(h_z\) induces a continuous functional on \(\F^{p_1}\) for any \(0<p_1<\infty\), thus again by Lemma \ref{lem:duality} and Lemma \ref{lem:inclusion}, \(h_z\in\F^{p_1'}\subset \F^\infty\). Therefore, we always have
\[h_z(w)=\e^{\bar{\alpha}\bar{z}^2+\bar{\eta}\bar{z}}\exp\bigl(\bar{\beta}w^2+
(\bar{\gamma}\bar{z}+\bar{\delta})w\bigr)\in \F^\infty\]
for any \(z\in\C\).
By Corollary \ref{cor:nonzero in fp}, we have \(|\beta|\le \frac{1}{2}\). Since \(\gamma\neq 0\) and \(z\) is arbitrary, we must have
\(|\beta|<\frac{1}{2}.\)

On the other hand, since
\[Tk_w(z)=\e^{-\frac{|w|^2}{2}+\beta\bar{w}^2+\delta\bar{w}}\exp\bigl(\alpha z^2+(\eta+\gamma\bar{w})z\bigr)\in\F^{q}\subset \F^\infty\]
for any \(w\in\C\). The same argument
gives \(|\alpha|<1/2\).
\end{proof}

By Lemma \ref{lem:necessary}, when \(T: \F^p\rightarrow \F^q\) is bounded,  all the conclusions hold in subsection \ref{subsec:action} and section \ref{sec:passage}.

The following fact about diagonal operators on sequence spaces is well-known.

\begin{lemma}\label{lem:D-bounded}
Let \(0<p, q\le \infty\). Then \(D: \ell^p(\Lambda)\to \ell^q(\Lambda)\) is bounded if and only if \(\{r_\lambda\}_{\lambda\in\Lambda}\in\ell^\infty(\Lambda)\) when \(p\le q\), and \(\{r_\lambda\}_{\lambda\in\Lambda}\in\ell^\rho(\Lambda)\) when \(q< p\), where \(1/\rho=1/q-1/p\) and \(1/\infty=0\).
\end{lemma}

\begin{proof}[Proof of Theorem~\ref{thm:fp-fq}]
If \(T\) is bounded, then
\[\sup_{w\in\C}\|Tk_w\|_q\le\|T\|,\] since
\(\|k_w\|_p=1\) for \(w\in\C\). With Lemma \ref{lem:necessary}, Lemmas~\ref{lem:kernels} and~\ref{lem: B-infty} therefore
give condition (ii) for \(0<q<\infty\) and \(q=\infty\) respectively.

Conditions (ii) and (iii) are equivalent by
Lemma~\ref{lem:R-C-equivalence}.

Assume (ii). Then \(\{r_\lambda=\mathcal R(\lambda)\}_{\lambda\in\Lambda}\) is bounded and thus \(D:\ell^p(\Lambda)\to\ell^q(\Lambda)\) is bounded by Lemma \ref{lem:D-bounded}.
By Proposition~\ref{prop:A=D},
\(A:\ell^p(\Lambda)\to\ell^q(\Lambda)\) is bounded.

If \(p<\infty\),
Proposition~\ref{prop:passage} gives boundedness of \(T: \F^p\to \F^q\).

If \(p=\infty\), then necessarily \(q=\infty\). By condition (iii),
\(\mathcal C\) is bounded.  By Lemma \ref{lem: B-integral}, for any \(f\in\F^\infty\),
\[
 |Tf(z)|e^{-|z|^2/2}
 \le \|f\|_\infty\frac1\pi\int_{\C}|B(z,w)| d A(w)
 =2C_{\beta, 1}\|f\|_\infty\mathcal C(z),\qquad z\in\C.
\]
Thus
\[\|Tf\|_\infty\le 2C_{\beta, 1}\sup_{z\in\C}\mathcal C(z)\|f\|_\infty\]
and \(T\) is bounded on \(\F^\infty\).
\end{proof}

\begin{lemma}\label{lem:dilation}
Let \(0<s<1\). For every \(0<p\le\infty\), define \(D_s\) on
\(\F^p\) as
\[(D_sf)(z)=f(sz),\qquad f\in \F^p.\] Then \(D_s\) is compact on
\(\F^p\) . On \(\F^2\), its singular values are \(1,s,s^2,\ldots\).
\end{lemma}

\begin{proof}
Let \(\{f_n\}\) be a bounded sequence in \(\F^p\). The continuous inclusion
\(\F^p\subset\F^\infty\) makes \(\{f_n\}\) locally bounded, so
Montel's Theorem provides a subsequence converging locally uniformly to an
entire function \(f\). For \(0<p<\infty\), by Fatou Lemma, we have
\[
\int_{\C}|f(z)\e^{-\frac{|z|^2}{2}}|^pd\nu_p(z)\le \limsup_{n\to\infty}\int_{\C}|f_n(z)\e^{-\frac{|z|^2}{2}}|^pd\nu_p(z)\le\sup_n\|f_n\|_p^p<\infty;
\]
for \(p=\infty\) and any \(z\in\C\),
\[
|f(z)\e^{-\frac{|z|^2}{2}}|\le \limsup_{n\to\infty}|f_n(z)\e^{-\frac{|z|^2}{2}}|\le \sup_n\|f_n\|_\infty<\infty.
\]
Therefore, for \(0<p\le\infty\), we always get \(f\in \F^p\) and \(\|f\|_p\le \sup_{n}\|f_n\|_p\).
Moreover, by Lemma \ref{lem:reproducing},
\[
\begin{aligned}
 |(D_sf_n-D_s f)(z)\e^{-|z|^2/2}|&=|f_n(sz)-f(sz)|e^{-|z|^2/2}
 \le \|f_n-f\|_p \e^{-(1-s^2)|z|^2/2}\\
 &\le 2^{\max\{1, 1/p\}}\sup_n\|f_n\|_p\e^{-(1-s^2)|z|^2/2} \quad (1/\infty=0).
 \end{aligned}
\]
Lebesgue dominated convergence theorem gives
\(D_sf_n\to D_sf\) in \(\F^p\) for \(0<p<\infty\) and local uniform convergence, \(\lim_{|z|\to\infty}\e^{-(1-s^2)|z|^2/2}=0\)   give \(D_sf_n\to D_sf\) in \(\F^\infty\).  Hence \(D_s\) is compact.

Since \(e_n(z)=z^n/\sqrt{n!}, n=0, 1, \dots, \) is an orthonormal basis of \(\F^2\) and
\(D_se_n=s^ne_n\), the singular-value statement of \(D_s\) follows.
\end{proof}

\begin{lemma}\label{lem:sufficient}
Assume \(\gamma\neq0\). Let \(0<p, q\le\infty\). If
\(|\alpha|<1/2, |\beta|<1/2\) and \(\mathcal R(w)\to 0\) as \(|w|\to\infty\),
then \(T: \F^p\to \F^q\) is bounded.
\end{lemma}
\begin{proof}
Let \(r_\lambda=\mathcal R(\lambda), \lambda\in\Lambda\). Then \(r\lambda\in\ell^\infty(\Lambda)\)and by Lemma \ref{lem:quadratic-integrability}, \(\{r_\lambda\}_{\lambda\in\Lambda}\in \ell^t(\Lambda)\) for any \(0<t<\infty\). Then for any  \(0<p, q\le\infty\), \(D: \ell^p(\Lambda)\to \ell^q(\Lambda)\) is bounded. By Proposition \ref{prop:A=D}, \(A: \ell^p(\Lambda)\to \ell^q(\Lambda)\) is bounded. When \(0<p<\infty\), we obtain that \(T: \F^p\to \F^q\) is bounded by Proposition \ref{prop:passage}.

Assume that \(p=\infty\). By Lemma \ref{lem:R-C-equivalence}, \(\mathcal C(z)\to 0\) as \(|z|\to\infty\). Obviously, \(\mathcal C\in L^\infty(\C, dA)\).  By Lemma
\ref{lem:quadratic-integrability}, \(\mathcal C\in L^t(\C, dA)\) for any \(0<t<\infty\). For any \(f\in\F^\infty\),
\[
 |Tf(z)|e^{-|z|^2/2}
 \le \|f\|_\infty\frac1\pi\int_{\C}|B(z,w)| d A(w)
 =2C_{\beta, 1}\|f\|_\infty\mathcal C(z),\qquad z\in\C.
\]
Hence, for any \(0<q\le \infty\),
\[
\|Tf\|_q\le 2C_{\beta, 1}C_q\|\mathcal C\|_{L^q(\C, dA)}\|f\|_\infty,
\]
where \(C_q=\bigl(\frac{q}{2\pi}\bigr)^{1/q}\) when \(0<q<\infty\), \(C_q=1\) when \(q=\infty\).
Thus \(T: \F^p\to \F^q\) is bounded.
\end{proof}

\begin{proposition}\label{prop:fp-fq-compact}
Assume \(\gamma\neq0\). Let \(0<p, q\le\infty\). Then the following are
equivalent:
\begin{enumerate}
\item[\textup{(i)}]  \(T:\F^p\to\F^q\) is compact;
\item[\textup{(ii)}]  \(|\alpha|<1/2,\quad  |\beta|<1/2\) and \(\mathcal R(w)\to 0\) as \(|w|\to\infty\);
\item[\textup{(iii)}]  \(|\alpha|<1/2,\quad  |\beta|<1/2\) and \(\mathcal C(z)\to 0\) as \(|z|\to\infty\).
\end{enumerate}
Equivalently, the strict inequalities \eqref{eq:strict-inequality} displayed in
Theorem~\ref{thm:fp-fq compact} hold.
\end{proposition}
\begin{proof}
Suppose \(T\) is compact. Then \(T\) is bounded. By Lemma \ref{lem:necessary}, we have \(|\alpha|<1/2, |\beta|<1/2\). If \(\lim_{|w|\to\infty}\mathcal R(w)\neq 0\), then by
Lemma~\ref{lem:quadratic-integrability}, there would be a positive constant \(\varepsilon_0\) and a lattice sequence
\(\lambda_n\to\infty\) for which \(r_{\lambda_n}\ge\varepsilon_0\).
Without loss of generality,
we may assume that all \(\lambda_n\) lie in \(\Lambda_j\) for some \(1\le j\le N\). For \(n\ne m\), let
\[\zeta=\{\zeta_\lambda\}: \zeta_{\lambda_n}=1, \zeta_{\lambda_m}=-1\quad \text{and}\quad \zeta_\lambda=0 \quad \text{for}\quad \lambda\neq \lambda_n, \lambda_m,\]
then by Remark \ref{rem:SpRp}, Lemma \ref{lem:dense} and \eqref{eq:RTS=A},
\[
 \|Tk_{\lambda_n}-Tk_{\lambda_m}\|_q\ge C_{r, q} \|R_q\bigl(T(k_{\lambda_n}-k_{\lambda_m})\bigr)\|_{\ell^q(\Lambda)}=C_{r, q} \|R_qTS_p\zeta\|_{\ell^q(\Lambda)}=C_{r, q} \|A\zeta\|_{\ell^q(\Lambda)},\]
 by Lemma \ref{lem:DlessA},
\[ \|A\zeta\|_{\ell^q(\Lambda)}\ge c\|\{r_\lambda \zeta_\lambda\}\|_{\ell^q(\Lambda_j)}\ge c|r_{\lambda_n}|
 \ge c\varepsilon_0.
\]
Thus \(\{Tk_{\lambda_n}\}\subset\F^q\) has no Cauchy
subsequence, contradicting the compactness of \(T\). Hence (ii) holds.

Conditions (ii) and (iii) are equivalent by
Lemma~\ref{lem:R-C-equivalence}.
Lemma~\ref{lem:quadratic-integrability} converts the condition (ii) and (iii) into the
strict inequalities \eqref{eq:strict-inequality}.

Conversely, assume that (ii) holds.  Choose \(0<s<1\), sufficiently
close to \(1\), such that
\begin{align*}
\frac{|\alpha|}{s^2}<\frac12,\qquad  \frac{|\gamma/s|^2}{1-4|\alpha/s^2|^2}
 +2\left|\bar\beta+
 \frac{\alpha\bar\gamma^2/s^4}{1-4|\alpha/s^2|^2}\right|<1.
\end{align*}
Define the integral operator \(T_s\) as
\begin{align*}
(T_{s}f)(z)=\frac{1}{\pi}\int_\C
  \exp\bigl(\frac{\alpha}{s^2}z^2+\beta\overline w^2
  +\frac{\gamma}{s}z\overline w+\frac{\eta}{s}z
  +\delta\overline w\bigr)
  f(w)\e^{-|w|^2}d A(w).
\end{align*}
Then by Lemma~\ref{lem:sufficient},
\(T_{s}: \F^p\to\F^q\) is bounded.

For every \(f\in\F^p\) and every \(z\in\C\),
\begin{align*}
  &(D_s T_{s}f)(z)=(T_{s}f)(sz)\\
  &=\frac{1}{\pi}\int_\C
  \exp\Bigl(\frac{\alpha}{s^2}(sz)^2+\beta\overline w^{\,2}
  +\frac{\gamma}{s}(sz)\overline w+\frac{\eta}{s}(sz)
  +\delta\overline w\Bigr)
  f(w)\,\e^{-|w|^2}d A(w)\\
  &=(Tf)(z),
\end{align*}
hence
\(
 T=D_sT_{s}.
\)
Since \(T_{s}:\F^p\to\F^q\) is bounded and \(D_s\) is compact on
\(\F^q\) by Lemma~\ref{lem:dilation}, \(T\) is compact.
\end{proof}

Theorem~\ref{thm:fp-fq compact} follows from Proposition \ref{prop:fp-fq-compact}.

\begin{proof}[Proof of Theorem \ref{thm:fp-fq-compact-q-less-p}]
(i) implies (ii).

Assume that \(T:\F^p\to\F^q\) is bounded. By Lemma \ref{lem:necessary}, we have \(|\alpha|<1/2, |\beta|<1/2\).
If \(0<p<\infty\), then by  Proposition  \ref{prop:passage}, \(A: \ell^p(\Lambda)\to \ell^q(\Lambda)\) is bounded.
Thus \(D: \ell^p(\Lambda)\to \ell^q(\Lambda)\) is bounded
by Proposition  \ref{prop:A=D}.
 Therefore, by Lemma \ref{lem:D-bounded},
\(
 \{r_\lambda\}\in\ell^\rho(\Lambda)\),
with
 \(\frac1\rho=\frac1q-\frac1p.
\)
Since
\(r_\lambda=e^{\Theta_\alpha(\lambda)}\),
Lemma~\ref{lem:quadratic-integrability} implies condition (iii).
If \(p=\infty\), then by Lemma \ref{lem:inclusion}, for any \(0<p_1<\infty\), \(T:\F^{p_1}\to\F^q\) is bounded. Choose \(q<p_1<\infty\), then condition (iii) follows as in the case \(0<p<\infty\).

Conditions (iii) and (iv) are equivalent by
Lemma~\ref{lem:R-C-equivalence}. Lemma~\ref{lem:quadratic-integrability} converts condition (iii) and (iv) into the
strict inequalities \eqref{eq:strict-inequality}.

By Proposition \ref{prop:fp-fq-compact}, the equivalence of (iii)  and (i) follows.
\end{proof}

\subsection{The case of \(\gamma=0\)}

Assume that \(\gamma=0\). As in the introduction, we define
\[
  P(z) = \exp(\alpha z^2+\eta z),\qquad z\in\C\quad
\text{and}\quad
  Q(w)=\exp(\bar\beta w^2+\bar\delta w),\qquad w\in\C.
\]

\begin{proof}[Proof of Theorem \ref{thm:gamma=0}]
Since \(\gamma=0\),
\[H(z, w)=P(z)\overline{Q(w)}\]
and then
\begin{align*}
(Tf)(z)=&\frac1\pi\int_{\C}H(z,w)f(w)e^{-|w|^2}dA(w)\\
=&\frac{1}{\pi}\int_{\C}P(z)\overline{Q(w)}f(w)e^{-|w|^2}dA(w)=\langle f, Q\rangle P(z).\end{align*}
Therefore, \(T: \F^p\to \F^q\) is bounded if and only if \(P\in\F^q\) and \(Q\) induces a continuous functional on  \(\F^p\).

When \(0<p<\infty\), it follows from  Lemma \ref{lem:duality} that \(Q\) induces a bounded functional on  \(\F^p\) if and only if \(Q\in\F^{p'}\).
If \(0<p\le 1\), then \(p'=\infty\). If \(1<p<\infty\), then \(1<p'<\infty\), by Corollary \ref{cor:nonzero in fp}, \(Q\in\F^{p'}\) if and only if \(|\beta|<1/2\).

If \(p=\infty\) and \(Q\) induces a bounded functional on \(\F^\infty\), then by Lemma \ref{lem:inclusion},  \(Q\) induces a bounded functional on \(\F^{p_1}\) for any \(0<p_1<\infty\). Choose \(p_1>1\), then by the case \(1<p<\infty\), we obtain \(|\beta|<1/2\). On the other hand, if \(|\beta|<1/2\), then \(Q\in\F^1\), by Lemma \ref{lem:duality}, it is easy to verify that \(Q\) induces a bounded functional on \(\F^\infty\). Therefore, \(Q\) induces a bounded functional on \(\F^\infty\) if and only if \(|\beta|<1/2\).

Therefore, the equivalence of (ii) and (iii) follows.

Since
\(T=P\otimes Q,\)
thus, for all \(0<p,q\le\infty\),  if \(T: \F^p\to \F^q\) is bounded, then \(T\) is rank-one operator.  Hence the equivalence of (i) and (ii) is obvious.
\end{proof}

\section{Schatten class on $\F^2$}

In this section, we will show that if the integral operator \(T\) is compact on \(\F^2\), then \(T\in S_p(\F^2)\)  for every \(0<p<\infty\), where \(S_p(\F^2)\) denotes the Schatten \(p\)-class of \(\F^2\).

\begin{proposition}\label{prop:S2}
Assume \(|\alpha|<1/2\) and \(|\beta|<1/2\).
If
\[
 \frac{|\gamma|^2}{1-4|\beta|^2}
 +2\left|\alpha+
 \frac{\bar\beta\gamma^2}{1-4|\beta|^2}\right|<1,
\]
then
\(T\in S_2(\F^2)\), and
\[
  \|T\|_{S_2(\F^2)}^2
  =\frac{C_{\beta, 2}^2}{\sqrt{A_\beta^2-4|N_\beta|^2}}\,
   \exp\left(
   \frac{-A_\beta|M_\beta|^2
         +\overline{N_\beta}\,M_\beta^2
         +N_\beta\,\overline{M_\beta}^{\,2}}
        {A_\beta^2-4|N_\beta|^2}
   \right),
\]
where \(A_\beta=|\gamma|^2/(1-4|\beta|^2)-1\), and \(N_\beta, M_\beta\) as in \eqref{eq:NbetaMbeta}, \(C_{\beta, 2}\) as in Lemma \ref{lem: B-integral}.

\end{proposition}

\begin{proof}
Let \(\{e_n\}\), \(e_n(z)=\sqrt{1/n!}z^n\), be the standard
orthonormal basis of \(\F^2\). For every \(z\in\C\),
\[
  (Tf)(z)=\ip{f}{h_z},
\]
where
\(h_z(w)=\overline{H(z, w)}=\e^{\overline\alpha\,\overline z^{\,2}+\overline\eta\,\overline z}\,
  \e^{\overline\beta\,w^2+(\overline\gamma\,\overline z+\overline\delta)w}\), and
\(h_z\in\F^2\) since \(|\overline\beta|<1/2\). By Parseval's identity,
\[
  \sum_{n\ge0}|Te_n(z)|^2
  =\sum_{n\ge0}\big|\ip{e_n}{h_z}\big|^2
  =\|h_z\|_2^2 ,
\]
and therefore,
\begin{equation}\label{eq:S2norm}
  \sum_{n\ge0}\|Te_n\|_2^2
  =\frac1\pi\int_\C\|h_z\|_2^2\,\e^{-|z|^2}d A(z).
\end{equation}
By Lemma \ref{lem: B-integral},
\begin{align*}
\|h_z\|_2^2=&\frac1\pi\int_{\C}|h_z(w)e^{-\frac{|w|^2}{2}}|^2dA(w)=\frac1\pi\int_{\C}|H(z,w)e^{-\frac{|w|^2}{2}}|^2dA(w)\\
=&\frac1\pi\int_{\C}|B(z,w)e^{\frac{|z|^2}{2}}|^2dA(w)=e^{|z|^2}\frac1\pi\int_{\C}|B(z,w)|^2dA(w)\\
=& C_{\beta, 2}^2e^{|z|^2}\mathcal C^2(z).
\end{align*}
Hence \eqref{eq:S2norm} becomes
\[
  \sum_{n\ge0}\|Te_n\|_2^2
  =\frac{C_{\beta, 2}^2}{\pi}
   \int_\C\mathcal C^2(z)d A(z).
\]
Since
\[\mathcal C^2(z)=\e^{2\Theta_\beta(z)}=\exp\bigl(A_\beta|z|^2+2\Real(N_\beta z^2+M_\beta z)\bigr),\qquad z\in\C,\] here \(A_\beta=|\gamma|^2/(1-4|\beta|^2)-1\),
applying Lemma~\ref{lem:Gaussian integral}, we obtain
\[
  \frac1\pi\int_\C\mathcal C^2(z)d A(z)
  =\frac{1}{\sqrt{A_\beta^2-4|N_\beta|^2}}\,
  \exp\left(
  \frac{-A_\beta|M_\beta|^2+2\Real(\overline{N_\beta}\,M_\beta^2)}{A_\beta^2-4|N_\beta|^2}\right)<\infty .
\]
Thus \(\sum_n\|Te_n\|_2^2<\infty\), so \(T\) is a Hilbert--Schmidt operator and
\[
  \|T\|_{S_2(\F^2)}^2
  =\frac{C_{\beta, 2}^2}{\sqrt{A_\beta^2-4|N_\beta|^2}}\,
   \exp\left(
   \frac{-A_\beta|M_\beta|^2
         +\overline{N_\beta}\,M_\beta^2
         +N_\beta\,\overline{M_\beta}^{\,2}}
        {A_\beta^2-4|N_\beta|^2}
   \right).\qedhere
\]
\end{proof}

\begin{theorem}\label{thm:schatten}
The following statements are equivalent.
\begin{enumerate}
\item[{{\rm (i)}}] \(T: \F^2\to\F^2\) is compact.
\item[{{\rm (ii)}}] \(T\in S_p(\F^2)\) for every \(0<p<\infty\).
\end{enumerate}
\end{theorem}

\begin{proof}
If \(\gamma=0\), by {\it the proof of Theorem \ref{thm:gamma=0}}, \(T\) is a rank-one operator, hence  \(T\in S_p(\F^2)\) for every \(0<p<\infty\).

 Assume \(\gamma\neq 0\) and \(T\) is compact, then by Theorem \ref{thm:fp-fq compact}, the
two strict inequalities \eqref{eq:strict-inequality} hold. As in the proof of Proposition \ref{prop:fp-fq-compact},
we can choose  \(0<s<1\) such that \(T_s: \F^2\to \F^2\) is bounded and
\(T=D_sT_{s}.\)
By Lemma~\ref{lem:dilation} and the standard singular value inequality, we have
\[
  s_n(T)=s_n(D_s T_{s})\le s_n(D_s)\|T_{s}\|
  =\|T_{s}\|s^n,
  \qquad n=0,1,2,\ldots,
\]
here \(s_n(T), n=0,1,2,\ldots,\) denotes the singular values (counted with multiplicity) of \(T\) arranged in decreasing order.
Hence
\[
  \sum_{n=0}^\infty s_n(T)^p
  \le
  \|T_{s}\|^p\sum_{n=0}^\infty s^{np}<\infty
\]
for every \(0<p<\infty\). Thus \(T\in S_p(\F^2)\) for every \(0<p<\infty\).
\end{proof}


\begin{thebibliography}{99}

\bibitem{Bargmann}
V.~Bargmann,
\emph{On a Hilbert space of analytic functions and an associated integral
transform, Part~I},
Comm. Pure Appl. Math. \textbf{14} (1961), no.~3, 187--214.

\bibitem{DZ1}
X.~Dong and K.~Zhu,
\emph{Canonical integral operators on the Fock space},
Math. Z. \textbf{306} (2024), Paper No.~64.

\bibitem{DZ2}
X.~Dong and K.~Zhu,
\emph{Canonical integral operators on the Fock space II},
arXiv:2509.16672 [math.FA], 2025.

\bibitem{Folland}
G.~B. Folland,
\emph{Harmonic Analysis in Phase Space},
Annals of Mathematics Studies, vol.~122,
Princeton University Press, Princeton, NJ, 1989.

\bibitem{HK}
P.~V. Hai and L.~H. Khoi,
\emph{Complex symmetry of weighted composition operators on the Fock space},
J. Math. Anal. Appl. \textbf{433} (2016), 1757--1771.

\bibitem{Howe}
R.~Howe,
\emph{The oscillator semigroup},
in \emph{The Mathematical Heritage of Hermann Weyl}
(Durham, NC, 1987),
Proc. Sympos. Pure Math., vol.~48,
American Mathematical Society, Providence, RI, 1988, pp.~61--132.







\bibitem{Lieb}
E.~H. Lieb,
\emph{Gaussian kernels have only Gaussian maximizers},
Invent. Math. \textbf{102} (1990), no.~1, 179--208.

\bibitem{Neretin} Yurii A. Neretin,  \emph{Lectures on
Gaussian integral operators and classical groups}, European Mathematical Society, 2011.



\bibitem{SZ}
Y. ~Sang and L.~Zhao, \emph{Complex symmetric weighted composition operators on the Fock space}, preprint.

\bibitem{Zhao2014}
L.~Zhao,
\emph{Unitary weighted composition operators on the Fock space of $\C^n$},
Complex Anal. Oper. Theory \textbf{8} (2014), 581--590.

\bibitem{Zhao2017}
L.~Zhao, \emph{Normal weighted composition operators on the Fock space of \(C^N\)} , Operators and Matrices \textbf{11(3)}(2017): 697-704.



\bibitem{Zhu2012}
K.~Zhu,
\emph{Analysis on Fock Spaces},
Graduate Texts in Mathematics, vol.~263,
Springer, New York, 2012.



\end{thebibliography}
\end{document}